%% file: main.tex
\documentclass[letterpaper,notitlepage,11pt,dvipsnames]{amsart}

\title{Realizing Transfer Systems as Suboperads of Coinduced Operads}
\author{Ben Szczesny} 
\date{}

\usepackage[margin=1.15in]{geometry} %
\usepackage[english]{babel}
\usepackage[utf8]{inputenc}

\usepackage{lmodern} %
\usepackage{amsmath, amssymb, amsthm, mathtools} %
\usepackage[bb=dsserif,bbscaled=1.1,scr=euler,cal=dutchcal]{mathalpha}
\usepackage[pdfencoding=auto, psdextra]{hyperref} %
\hypersetup{
    colorlinks=true,
    linkcolor=black,
    filecolor=magenta,      
    urlcolor=black,
    citecolor=black,
    }

\usepackage{cancel}

\usepackage[inline]{enumitem} %
\setlist[enumerate,1]{nosep, label = (\arabic*), labelindent=\parindent}

\usepackage{fancyhdr} 
\fancypagestyle{draft}{%
\fancyhf{}%
\fancyhead[R]{Compiled: \today}
\fancyfoot[C]{\thepage}%
}
\usepackage{bookmark}%
\numberwithin{equation}{section}

\usepackage[nameinlink,noabbrev]{cleveref}
\crefname{diagram}{diagram}{diagrams} %
\Crefname{diagram}{Diagram}{Diagrams}
\creflabelformat{diagram}{#2(#1)#3}
\crefname{condition}{condition}{conditions}
\Crefname{condition}{Condition}{Conditions}
\creflabelformat{condition}{#2#1#3}

\usepackage{thmtools}

\newtheorem{mainthm}{Theorem}

\swapnumbers
\declaretheorem[name=Theorem,
refname={theorem,theorems},
Refname={Theorem,Theorems},
parent=section]{theorem}

\declaretheorem[name=Lemma,
refname={lemma,lemmas},
Refname={Lemma,Lemmas},
parent=section,
sibling=theorem]{lemma}

\declaretheorem[name=Definition,
refname={definition,definitions},
Refname={Definition,Definitions},
parent=section,
style=definition,
sibling=theorem]{definition}

\declaretheorem[name=Notation,
refname={notation,notations},
Refname={Notation,Notations},
parent=section,
style=definition,
sibling=theorem]{notation}

\declaretheorem[name=Example,
refname={example,examples},
Refname={Example,Examples},
parent=section,
style=definition,
sibling=theorem]{example}

\declaretheorem[name=Remark,
refname={remark,remarks},
Refname={Remark,Remarks},
parent=section,
style=remark,
sibling=theorem]{remark}

\declaretheorem[name=Warning,
refname={warning,warnings},
Refname={Warning,Warnings},
parent=section,
style=remark,
sibling=theorem]{warning}

\usepackage{graphicx} 
\usepackage[x11names]{xcolor}
\usepackage{tikz}
\usetikzlibrary{fit}
\usetikzlibrary{positioning}
\usetikzlibrary{matrix}
\usetikzlibrary{calc}
\usetikzlibrary{intersections}
\usetikzlibrary{math}
\usetikzlibrary{shapes.geometric, arrows.meta}
\usetikzlibrary{external}
\usepackage{quiver}

\usepackage[backend=biber,style=alphabetic,giveninits]{biblatex}
\usepackage{csquotes} %
\usepackage{todonotes}
\makeatletter
\renewcommand{\todo}[2][]{\tikzexternaldisable\@todo[#1]{#2}\tikzexternalenable}
\makeatother

\input{./content/syntax.tex}

\begin{document}
\begin{abstract}
    In this paper, we present an explicit method to identify equivariant suboperads of coinduced operads that contain only fixed points associated to any desired transfer system. Our method works for a class of operads that we call intersection operads, which includes many familiar operads of interest, including the little $k$-cube operads, the Steiner operad, and the linear isometries operad. As an application, we also construct an intersection $\mathbb{E}_\infty$-operad that, when applying our construction, will produce a $\mathbb{N}_\infty$-operad realizing an arbitrary transfer system.
\end{abstract}
\maketitle
\input{content/introduction.tex}

\input{content/intersectionmonoids.tex}

\input{content/incompleteindexingoperad.tex}

\input{content/realizingtransfers.tex}

\input{content/applications.tex}

\appendix
\input{content/appendix.tex}

\printbibliography
\addcontentsline{toc}{chapter}{Bibliography}
\end{document}

%% file: content/syntax.tex
\newcommand\restr[2]{{%
  \left.\kern-\nulldelimiterspace %
  #1 %
  \littletaller %
  \right|_{#2} %
  }}
\newcommand{\littletaller}{\mathchoice{\vphantom{\big|}}{}{}{}}

\DeclarePairedDelimiterX\set[1]\lbrace\rbrace{\def\given{\;\delimsize\vert\;}#1}%

\DeclarePairedDelimiter\abs{\lvert}{\rvert}
 
\DeclarePairedDelimiterX\inner[1]\langle\rangle{\def\and{\,\delimsize\vert\,}#1}

\usepackage{semantex}
\usepackage{stmaryrd}

\newcommand{\Hill}{Hill}

\newcommand{\Blumberg}{Blumberg}

\NewVariableClass\MyVar[output=\MyVar]
\NewObject\MyVar\va {a}
\NewObject\MyVar\vb {b}
\NewObject\MyVar\vc {c}
\NewObject\MyVar\vd {d}
\NewObject\MyVar\ve {e}
\NewObject\MyVar\vf {f}
\NewObject\MyVar\vg {g}
\NewObject\MyVar\vh {h}
\NewObject\MyVar\vi {i}
\NewObject\MyVar\vj {j}
\NewObject\MyVar\vk {k}
\NewObject\MyVar\vl {l}
\NewObject\MyVar\vm {m}
\NewObject\MyVar\vn {n}
\NewObject\MyVar\vo {o}
\NewObject\MyVar\vp {p}
\NewObject\MyVar\vq {q}
\NewObject\MyVar\vr {r}
\NewObject\MyVar\vs {s}
\NewObject\MyVar\vt {t}
\NewObject\MyVar\vu {u}
\NewObject\MyVar\vv {v}
\NewObject\MyVar\vw {w}
\NewObject\MyVar\vx {x}
\NewObject\MyVar\vy {y}
\NewObject\MyVar\vz {z}
\NewObject\MyVar\vA {A}
\NewObject\MyVar\vB {B}
\NewObject\MyVar\vC {C}
\NewObject\MyVar\vD {D}
\NewObject\MyVar\vE {E}
\NewObject\MyVar\vF {F}
\NewObject\MyVar\vG {G}
\NewObject\MyVar\vH {H}
\NewObject\MyVar\vI {I}
\NewObject\MyVar\vJ {J}
\NewObject\MyVar\vK {K}
\NewObject\MyVar\vL {L}
\NewObject\MyVar\vM {M}
\NewObject\MyVar\vN {N}
\NewObject\MyVar\vO {O}
\NewObject\MyVar\vP {P}
\NewObject\MyVar\vQ {Q}
\NewObject\MyVar\vR {R}
\NewObject\MyVar\vS {S}
\NewObject\MyVar\vT {T}
\NewObject\MyVar\vU {U}
\NewObject\MyVar\vV {V}
\NewObject\MyVar\vW {W}
\NewObject\MyVar\vX {X}
\NewObject\MyVar\vY {Y}
\NewObject\MyVar\vZ {Z}
\NewObject\MyVar\valpha {\alpha}
\NewObject\MyVar\vvaralpha {\varalpha}
\NewObject\MyVar\vbeta {\beta}
\NewObject\MyVar\vgamma {\gamma}
\NewObject\MyVar\vdelta {\delta}
\NewObject\MyVar\vepsilon {\epsilon}
\NewObject\MyVar\vvarepsilon {\varepsilon}
\NewObject\MyVar\vzeta {\zeta}
\NewObject\MyVar\veta {\eta}
\NewObject\MyVar\vtheta {\theta}
\NewObject\MyVar\viota {\iota}
\NewObject\MyVar\vkappa {\kappa}
\NewObject\MyVar\vlambda {\lambda}
\NewObject\MyVar\vmu{\mu}
\NewObject\MyVar\vnu{\nu}
\NewObject\MyVar\vxi{\xi}
\NewObject\MyVar\vpi{\pi}
\NewObject\MyVar\vvarpi {\varpi}
\NewObject\MyVar\vrho {\rho}
\NewObject\MyVar\vsigma {\sigma}
\NewObject\MyVar\vtau {\tau}
\NewObject\MyVar\vupsilon {\upsilon}
\NewObject\MyVar\vphi {\phi}
\NewObject\MyVar\vvarphi {\varphi}
\NewObject\MyVar\vchi {\chi}
\NewObject\MyVar\vpsi {\psi}
\NewObject\MyVar\vomega {\omega}
\NewObject\MyVar\vGamma {\Gamma}
\NewObject\MyVar\vDelta {\Delta}
\NewObject\MyVar\vTheta {\Theta}
\NewObject\MyVar\vLambda {\Lambda}
\NewObject\MyVar\vXi{\Xi}
\NewObject\MyVar\vPi{\Pi}
\NewObject\MyVar\vSigma {\Sigma}
\NewObject\MyVar\vUpsilon {\Upsilon}
\NewObject\MyVar\vPhi {\Phi}
\NewObject\MyVar\vPsi {\Psi}
\NewObject\MyVar\vOmega {\Omega}
\NewObject\MyVar\RR {\mathbb{R}}
\NewObject\MyVar\id {\operatorname{id}}
\NewObject\MyVar\vell {\ell}

\NewObject\MyVar\Orth{\operatorname{O}}

\NewObject\MyVar\Domain{\operatorname{Dom}}
\NewObject\MyVar\Aut{\operatorname{Aut}}
\NewObject\MyVar\Hom{\operatorname{Hom}}
\NewObject\MyVar\Map{\operatorname{Map}}
\NewObject\MyVar\Fun{\operatorname{Fun}}
\NewObject\MyVar\Isotropy{\operatorname{Stab}}
\NewObject\MyVar\AdmissibleSetsOf{\mathcal{A}}
\NewObject\MyVar\PhantomIsotropy{\operatorname{Ph}}
\NewObject\MyVar\PhantomAdmissibleSet{\operatorname{PhA}}
\NewObject\MyVar\Coinduced{\operatorname{N}}
\NewObject\MyVar\QuasiCoinduced{\widetilde{\operatorname{N}}}
\NewObject\MyVar\Realize{\mathcal{R}}
\NewObject\MyVar\Image{\operatorname{im}}
\NewObject\MyVar\TransferSystemsOf{\operatorname{Tr}}
\NewObject\MyVar\VerticalEqualSet{\operatorname{E}}
\NewObject\MyVar\OrbitType{\operatorname{T}}
\NewObject\MyVar\SliceFactor{\operatorname{F}}
\NewObject\MyVar\SubgroupLattice{\operatorname{Sub}}
\NewObject\MyVar\HomotopyCategoryOf{\operatorname{Ho}}
\NewObject\MyVar\EquivariantSetOf{\operatorname{Set}}
\NewObject\MyVar\Forget{\operatorname{\shortdownarrow}}
\NewObject\MyVar\EqualSet{\operatorname{E}}
\newcommand{\RightCosets}[2]{#1 \backslash #2}

\NewObject\MyVar\Configurations{\operatorname{Conf}}
\NewObject\MyVar\Interior{\operatorname{Int}}
\NewObject\MyVar\IndexingSystemGeneratedBy{\mathbb{I}}
\NewObject\MyVar\GraphSubgroupsGeneratedBy{\mathbb{G}}
\NewObject\MyVar\Length{\operatorname{len}}
\NewObject\MyVar\GraphOf{\operatorname{\Gamma}}
\NewObject\MyVar\Shift{\operatorname{sh}}
\NewObject\MyVar\Combine{\operatorname{c}}
\NewObject\MyVar\ColumnType{\operatorname{T}}
\NewObject\MyVar\SetFromGraph{\mathcal{S}}
\NewObject\MyVar\FixedSubgroupsOf{\mathcal{F}}
\NewObject\MyVar\CoefficientSystemOf{\mathcal{C}}
\NewObject\MyVar\TransfersOf{\mathcal{T}}
\NewObject\MyVar\GeneratedCoefficientSystemOf{\mathcal{C}}
\NewObject\MyVar\TwistMaps{\operatorname{TMap}}
\newcommand{\Transfer}[1]{\xrightarrow[#1]{}}

\NewObject\MyVar\LittleCubeOperad{\mathcal{C}}
\NewObject\MyVar\LinearIsometriesOperad{\mathcal{L}}
\NewObject\MyVar\SteinerOperad{\mathcal{K}}
\NewObject\MyVar\EOperad{\mathbb{E}}
\NewObject\MyVar\NOperad{\mathbb{N}}
\NewObject\MyVar\OConstr{\vO[cal]}
\NewObject\MyVar\Assoc{\operatorname{Assoc}}
\NewObject\MyVar\BarrattEccles{\mathscr{E}}
\NewObject\MyVar\DyadicBarrattEccles{\mathscr{DE}}
\NewObject\MyVar\DyadicAssoc{\mathscr{D}\!\operatorname{Assoc}}
\NewObject\MyVar\IncompleteIndexOperad{\mathbb{I}}
\NewObject\MyVar\DyadicIntervalMonoid{\mathscr{D}}
\NewObject\MyVar\FatDyadicIntervalMonoid{\mathscr{WD}}
\NewObject\MyVar\LinearOrderOperad{\operatorname{Lin}}

\newcommand{\GSpaces}[1][G]{\operatorname{Top}^{#1}}
\newcommand{\Spaces}{\operatorname{Top}}
\newcommand{\Set}{\operatorname{Set}}
\newcommand{\GSets}[1][G]{\operatorname{Set}^{#1}}
\newcommand{\GFiniteSets}[1][G]{\operatorname{Set}_{fin}^{#1}}
\newcommand{\Poset}{\operatorname{Pos}}
\newcommand{\GPosets}[1][G]{\operatorname{Pos}^{#1}}
\newcommand{\SymSeq}{\operatorname{Sym}}
\newcommand{\GOperads}[1][G]{\operatorname{Oper}^{#1}}
\newcommand{\Operads}{\GOperads[]}
\newcommand{\OrbitCategory}[1][G]{\operatorname{O}_{#1}}
\NewObject\MyVar\CategoryOfIntersectionMonoids{\operatorname{Mon}^{\vee}}
\NewObject\MyVar\CategoryOfIntersectionOperads{\operatorname{Oper}^{\vee}}
\NewObject\MyVar\CategoryOfFiniteSetInjective{\textbf{Inj}}
\NewObject\MyVar\CategoryOfIndexingSystems{\operatorname{Ind}}
\NewObject\MyVar\CategoryOfCategories{\operatorname{Cat}}

\newcommand{\edge}[2]{#1 \text{\textemdash} #2}
\NewObject\MyVar\CompleteGraph{\mathcal{K}}

\SetupClass\MyVar{
define keys = {
  {cal}{command = \mathcal},
  {scr}{command = \mathscr},
  {indexsystem}{command = \mathscr},
  {~}{command = \widetilde},
  {finset}{command = \underline},
  {uline}{command = \underline},
  {oline}{command = \overline},
  {inv}{upper = -1},
  {vec}{command = \vec},
  {abs}{left spar ={\vert}, right spar ={\vert},spar},
  {graph}{command = \mathbf}
},
define keys[1] = {
  {pow}{upper = #1}
},
}

\NewSymbolClass\MyBinaryOperator [
define keys ={
{Lder }{ upper=L},
{Rder }{ upper=R},
{~}{command = \widetilde,command=\mathbin}
},
]

\NewObject\MyBinaryOperator\Tensor{\otimes}[
define keys ={
{der}{ Lder},
},
]

\NewObject\MyBinaryOperator\Disjoint{\vee}
\NewObject\MyBinaryOperator\Intersect{\wedge}
\NewObject\MyBinaryOperator\StrictIntersect{\mathbin{\underline{\wedge}}}
\NewObject\MyBinaryOperator\Asterick{\ast}
\NewObject\MyBinaryOperator\VerticalSum{\mathbin{+}}

%% file: content/introduction.tex
\section{Introduction}

Transfer and norm maps are essential algebraic structures in equivariant homotopy theory. These maps can be viewed as the extra structure that differentiates regular non-equivariant abelian groups and rings into ``genuine'' equivariant abelian groups and rings; more commonly known as Mackey and Tambara functors.

For the purposes of this paper, we will think of a norm map as a ``twisted multiplication'' in an equivariant $G$-category. To illustrate this view point, suppose we have a topological $G$-space $X$. If $X$ has the structure of a monoid in the category of topological $G$-spaces $\GSpaces$ with multiplication $\mu$, then $X$ comes packaged with a system of multiplication maps such as \[X\times X \times X \xrightarrow{\mu(\mu,\id)} X\] which are $G$-equivariant, where $G$ acts via the diagonal on the product $X\times X \times X$. However, in equivariant categories we also have ``twisted products'' or normed spaces \[N_K^HX = \prod_{H/K}i^\ast_KX.\] A norm map is then a $G$-equivariant map of the form \[G\times_HN_K^HX\to X.\]

Non-equivariantly, we can use the theory of $\mathbb{E}_\infty$-operads to capture a homotopy coherent system of multiplication maps on an object $X$; however, these operads can't encode norm maps. Blumberg-Hill \cite{blumbergOperadicMultiplicationsEquivariant2015} defined the notion of $\mathbb{N}_\infty$-operads as an extension of $\mathbb{E}_\infty$-operads that also encode a homotopy coherent system of norm maps via fixed points of the operad. 
An $\mathbb{N}_\infty$-operad may only encode a portion of all possible norm maps, and each possibility can be described by an algebraic object that \Blumberg-\Hill~call an indexing system. Rubin \cite{rubinDetectingSteinerLinear2021} and Balchin-Barnes-Roitzheim \cite{balchinNinftyoperadsAssociahedra2021} independently defined an equivalent notion called a transfer system which we will review in \cref{section: background} below. Transfer systems of $G$ form a lattice which we will denote by $\TransferSystemsOf{G}$. This assignment of a transfer system from an $\mathbb{N}_\infty$-operad is functorial, and in fact, turns out to be a homotopy invariant of $\mathbb{N}_\infty$-operads. \Blumberg-\Hill~proved we get a fully faithful functor of the form \begin{equation}\label{equation: Ninf homotopy to transfer}
    \TransfersOf:\HomotopyCategoryOf{\mathbb{N}_\infty\text{-operads}} \to  \TransferSystemsOf{G}. 
\end{equation}
They also conjectured this functor is essentially surjective, which has since been proven independently in three different ways \cite{gutierrezEncodingEquivariantCommutativity2018,bonventreGenuineEquivariantOperads2021,rubinCombinatorial$N_infty$Operads2021}. 

One application of the main result of this paper will be to give yet another proof of this conjecture; however, our focus will be more broad than these previous papers as we will not deal exclusively with $\mathbb{N}_\infty$-operads. The functor $\TransfersOf$ works more generally on (most) $\Sigma$-free $G$-operads, and we will say that a $G$-operad $P$ \emph{realizes} a transfer system $\tau$ if $\TransfersOf{P}=\tau$. The main result of this paper is a method to construct general $G$-operads that realize an arbitrary transfer system. Our approach will be explicit, and allow us to circumvent the use of heavier theoretical machinery such as cofibrant resolutions, bar or universal space constructions that the previous approaches employed.

The starting point for us is the following construction:
\begin{definition}
    Given a non-equivariant operad $P$, the coinduced operad $\Coinduced[e,pow=G]P$ is given on its components by \[\Coinduced[e,pow=G]P(n) = \Map{G,P(n)}.\] This is an operad since $\Map(G,-)$ is a right adjoint and so the defining operad diagrams commute with this functor.
\end{definition} 
If $P$ is $\Sigma$-free, one can show that $\Coinduced[e,pow=G]P$ realizes the complete transfer system of $G$. A natural idea to realize an arbitrary transfer system is to identify suboperads of coinduced operads $\Coinduced[e,pow=G]P$ that have the correct fixed points.

A counter example by Bonventre \cite[Example B.2.1.]{bonventreComparisonModelsEquivariant2017} shows that this approach is not always possible. Nevertheless, we will identify a class of operads, which we will call \emph{intersection operads}, where this approach does work. Moreover, many familiar operads such as the little $k$-cubes operad $\LittleCubeOperad[k]$, Steiner operads $\SteinerOperad_V$, and the linear isometry operad $\LinearIsometriesOperad$ belong to this class. The first main result of this paper is the following.

\begin{mainthm}\label{theorem A}
    For an intersection operad $P$ and transfer system $\tau$, there exists a $G$-suboperad $\Coinduced[e,pow=\tau,uline]{P}$ of $\Coinduced[e,pow=G]{P}$, which we will call the \emph{$\tau$-incomplete coinduction of $P$}, that realizes the transfer system $\tau$.
\end{mainthm}

Given this result, it is possible to then use these operads to construct $\mathbb{N}_\infty$-operads by following the approach taken by Rubin in \cite{rubinCombinatorial$N_infty$Operads2021}.
There is a functor $E$ that takes a $G$-operad in sets (with correct fixed points) to an $\mathbb{N}_\infty$-operad with the same fixed points. 
We could then realize a transfer system $\tau$ as an $\mathbb{N}_\infty$-operad as the composite $E\Coinduced[e,pow=\tau,uline]{P}$. 
The second main result of this paper shows that this is unnecessary. 

\begin{mainthm}\label{theorem B}
    Let $\tau$ be an arbitrary transfer system. There exists an intersection $\mathbb{E}_\infty$-operad $\DyadicBarrattEccles$, which we call the \emph{dyadic Barratt-Eccles operad}, such that $\Coinduced[pow=\tau,e,uline]{\DyadicBarrattEccles}$ is an $\mathbb{N}_\infty$-operad that realizes the transfer system $\tau$.
\end{mainthm}

\subsection{Organization of the paper}
We will start this paper in \cref{section: background} with relevant background on how fixed points of operads encode norm maps, and the equivalence between indexing and transfer systems. This will serve to introduce needed notation, but also explain how this equivalence works broadly for $G$-operads and not just $\mathbb{N}_\infty$-operads as is the focus of the current literature. We will also end this section with examples of what fixed points of coinduced operads look like, and motivate our general construction. i.e., we will illustrate why we focus on the idea of ``intersection'' in our operads, and how this will lead us to be able to successfully control what fixed points can appear when we take suboperads.

In \cref{section: intersection monoids and operads} we will define intersection monoids and operads. In essence, an intersection operad $P$ is completely determined by its unary monoid $P(1)$ which comes equipped with a notion of ``intersecting elements''. Once we have defined these, we will define a $G$-poset valued operad called the $\tau$-incomplete indexing operad $\IncompleteIndexOperad[pow=\tau]$ in \cref{section: indexing operad} which will be used to keep track of how elements in an intersection operad might self-intersect, and hence serve as a tool to control what possible fixed points can appear. We will then use $\IncompleteIndexOperad[pow=\tau]$ in \cref{section: realizing transfer systems} to construct the $\tau$-incomplete coinduction functor $\Coinduced[e,pow=\tau,uline]$ on intersection operads. 

Finally, in \cref{section: realizing ninf} we will construct the dyadic Barratt-Eccles operad $\DyadicBarrattEccles$ and prove that $\Coinduced[pow=\tau,e,uline]{\DyadicBarrattEccles}$ is an $\mathbb{N}_\infty$-operad that realizes the transfer system $\tau$.

\subsection{Acknowledgements}
The original seeds of this paper were a part of the author's thesis, and we would like to thank Mike Hill for his support and guidance during that time.
We would also like to thank the members of the author's group in the AMS Mathematical Research Community session on homotopical combinatorics: David Chan, Myungsin Cho, David Mehrle, Pablo S. Ocal, Ang\'{e}lica Osorno, and Paula Verdugo. Their interest in this work gave the author the motivation to finish this paper.

\subsection{Notation and Conventions}
We write $G$ for a finite group, and $K\leq H \leq G$ for a selection of subgroups. 
Given $g\in G$, we write conjugate subgroups as $H^g=\vg[inv]Hg$. 
For a left $G$-set $T$, we denote the set of orbits by $\RightCosets{G}{T}$. 

$\GSpaces[]$ is the category of compactly generated weak Hausdorff spaces. A $G$-category is a functor category of the form $\Fun{G,\mathcal{C}}$ where we view $G$ as a category with one object, and $\mathcal{C}$ is a category. In this paper, we will use the following $G$-categories: the category of $G$-spaces $\GSpaces$; the category of $G$-posets $\GPosets$; the category of $G$-sets $\GSets$; and the category of finite $G$-sets $\GFiniteSets$. Moreover, we use the symmetric monoidal structure given by products on all of these.

For a category $\mathcal{C}$, we denote the category of \emph{left} symmetric sequences in $\mathcal{C}$ by $\SymSeq(\mathcal{C})$. We use left symmetric sequences so that if $\mathcal{C}$ is a $G$-category, and $P\in \SymSeq(\mathcal{C})$, then the level components $P(n)$ have a left $G\times \Sigma_n$-action. We write $\Operads(\mathcal{C})$ for the category of (left) symmetric operads in $\mathcal{C}$. A morphism of $\GSpaces$-operads $\phi: P\to Q$ is a (weak) homotopy equivalence if for each $n$, the component map $\phi(n):P(n)\to Q(n)$ is a (weak) $G\times \Sigma_n$-homotopy equivalence. We will always use left actions in this paper. This will cause a slightly different convention than what is commonly used when dealing with mapping spaces. Given a left $\Sigma_n$-space $X$, we will view the mapping spaces $\Map(G,X)$ as a left $G\times \Sigma_n$-space where the action is given by $\left[(g,\sigma)\cdot f\right](h) = \sigma f(\vg[inv]h).$ Note the left action via the inverse on the domain.

Given $n\in \mathbb{N}$, we will write $\vn[finset]$ for the finite set $\set{1,2,\dots,n}$, where $\underline{0}=\emptyset$. We will primarily use Markl's formulation of operads (see for instance \cite{marklOperadsPROPs2008}) in terms of $\circ_i$-compositions. An operad $P$ in a  symmetric monoidal category $\vC[cal]$ is then given by a collection of morphisms \begin{align*}
    \circ_i: P(n)\times P(m) \to P(n+m-1)
\end{align*} for each $i\in \vn[finset]$ that satisfy unital, associativity, and equivariance conditions. For convenience, we will record these conditions here. Since we will only work with categories whose objects have underlying sets, we will state these conditions in equational form.

The unital condition says that there exists a unit $u\in P(1)$ such that for all $x\in P(n)$, and $i\in \vn[finset]$ we have \begin{gather}\label{eq:operad_unitality}
x\circ_i u = x \text{ and } u\circ_1 x = x.    
\end{gather}

The associativity condition then says that for $x\in P(n)$, $y\in P(m)$, $z\in P(\ell)$ and $i\in\vn[finset]$, $j\in \underline{n+m-1}$ we have 
\begin{gather}\label{eq:operad_associativity}
    (x\circ_i y)\circ_j z = \begin{cases*}
        (x \circ_j z) \circ_{i+\ell-1} y & if $1\leq j < i$, \\ 
        x\circ_i (y \circ_{j-i+1} z) & if $i\leq j \leq i+m-1$, \\
        (x\circ_{j-m+1} z) \circ_{i} y & if $i+m\leq j \leq n+m-1$.  
    \end{cases*}
\end{gather}

Given permutations $\sigma\in \Sigma_n$, and $\tau\in \Sigma_m$, the permutation $\sigma\circ_i\tau$ is given by 
    \[(\sigma\circ_i\tau)(k) = \begin{cases*}
        \sigma(k) & if $\sigma(k)<\sigma(i)$ and $k<i$, \\ 
        \sigma(k)+m-1 & if $\sigma(k)\geq\sigma(i)$ and $k<i$, \\ 
        \sigma(k-m+1) & if $\sigma(k)<\sigma(i)$ and $k > i+m-1$, \\ 
        \sigma(k-m+1) + m - 1 & if $\sigma(k)\geq \sigma(i)$ and $k > i+m-1$, \\ 
        \tau(k-i+1) + \sigma(i) - 1 & if $k\in \set{i,i+1,\dots,i+m-1}$.
    \end{cases*}\]
i.e., $\sigma\circ_i\tau$ is the permutation where we expand the $i$-th position of $\sigma$ out to $m$ elements and apply $\tau$. The equivariance condition is then
\begin{gather}
(\sigma\cdot \vx)\circ_{\sigma(i)} (\tau \cdot \vy) = (\sigma\circ_i\tau)\cdot (\vx\circ_i \vy). \label{eq:operad_equivariance}
\end{gather} Note that this condition is slightly different to that in \cite[Definition 11]{marklOperadsPROPs2008} since our $\Sigma_n$-actions are on the left.

\section{Background and Motivation}\label{section: background}

\subsection{Norm Maps and Transfers}

A $G$-operad encodes norm maps via its fixed points. As an illustration, suppose we have a operad $P\in\Operads(\GSpaces)$. Denote a generator of $C_3$ by $\tau$ and consider the subgroup of $C_3\times \Sigma_3$ generated by $(\tau,(1,2,3))$, which we denote by $\Gamma$. 
Note that $\Gamma$ is the graph of the homomorphism $\phi:C_3\to \Sigma_3$ determined by $\tau\mapsto (1,2,3)$.
If $X\in \GSpaces[C_3]$ is a $P$-algebra, then we have a $C_3\times \Sigma_3$-equivariant map \[P(3)\to \Hom{X^3,X}\] where, on the right, $C_3$ acts by conjugation and $\Sigma_3$ acts by permuting factors. Note, that here we are using that $\GSpaces$ is enriched over itself and $\Hom{X^3,X}$ is the space of all continuous maps. Taking fixed points we get that \[P(3)^\Gamma\to \Hom{X^3,X}[pow=\Gamma]\] and unpacking the condition on an element $f\in \Hom{X^3,X}[pow=\Gamma]$,we see that this is a continuous map $f:X^3\to X$ such that \begin{align*}
    ((\tau,(1,2,3))\cdot f)(x_1,x_2,x_3) &= f(x_1,x_2,x_3) \\
    \tau f((\vtau[inv],(1,2,3)^{-1})(x_1, x_2,x_3)) &= f(x_1,x_2,x_3), \shortintertext{which is equivalently}
    f\left((\tau,(123))\cdot(x_1,x_2,x_2)\right)&=\tau f(x_1,x_2,x_3).
\end{align*} This means that $f$ is a $C_3$-equivariant map \[f: \prod_{C_3}X\to X.\] i.e., a norm map. We therefore see that the fixed points $P(3)^\Gamma$ parameterize a family of norm maps of the form $\prod_{C_3}X\to X$.

Using this observation, \Blumberg-\Hill~extended $\mathbb{E}_\infty$-operads to include norm data by defining $\mathbb{N}_\infty$-operads as follows.
\begin{definition}[\protect{\cite[Definition 3.7]{blumbergOperadicMultiplicationsEquivariant2015}}]\label{definition:Ninf operads}
    An $\mathbb{N}_\infty$-operad $\mathcal{O}$ if an operad valued in $\GSpaces$ such that 
    \begin{enumerate}
        \item The space $\mathcal{O}(0)$ is contractible,
        \item The action of $\Sigma_n$ if free on $\mathcal{O}(n)$, and 
        \item\label[condition]{definition: Ninf condition three} $\mathcal{O}(n)$ is a universal space for a family $\mathcal{F}_n(\mathcal{O})$ of subgroups of $G\times \Sigma_n$ which contains all subgroups of the form $H\times \set{1}$. 
    \end{enumerate}
\end{definition}
Not every collection of family subgroups $\mathcal{F}_n(\mathcal{O})$ are possible. Because we require that $\Sigma_n$ act freely, each subgroup $\Gamma\in \mathcal{F}_n(\mathcal{O})$ must be a \emph{graph subgroup}. A graph subgroup $\Gamma \subseteq G\times \Sigma_n$ is a subgroup such that there is some subgroup $H \leq G$, and group homomorphism $\phi : H\to \Sigma_n$ such that \[\Gamma = \set{(h,\phi(\vh))\in G\times \Sigma_n\given h\in H}.\] Moreover, a group homomorphism $\phi:H\to \Sigma_n$ determines an $H$-set structure on $\vn[finset]$.

\begin{notation}
    Given a graph subgroup $\Gamma \subseteq G\times \Sigma_n$, we will use the following notation:
    \begin{enumerate}
        \item $H_\Gamma = \operatorname{pr}_{G}\Gamma$, the projection onto the $G$ component.
        \item $\phi_\Gamma$ for the corresponding homomorphism $H_\Gamma\to \Sigma_n$.
        \item $\SetFromGraph[\Gamma]$ for the corresponding $H_\Gamma$-set with underlying set $\vn[finset]$.
    \end{enumerate}
\end{notation}

Given a graph subgroup $\Gamma\in \mathcal{F}_n(\mathcal{O})$ is a graph subgroup, a map $f\in \mathcal{O}(n)^\Gamma$ corresponds to a norm map of the form \[f: G\times_H\prod_{\SetFromGraph[\Gamma]}X\to X\] on any $\mathcal{O}$-algebra $X$. Hence, \cref{definition: Ninf condition three} in \Cref{definition:Ninf operads} can be interpreted as saying that the space of all norm maps that an operad encodes is homotopy coherent. The family of graph subgroups $\set{\mathcal{F}_n(\mathcal{O})}_n$ is establishing exactly which norm maps an $\mathbb{N}_\infty$-operad encodes.

\subsection{Indexing and Transfer Systems}

The operadic structure of an $\mathbb{N}_\infty$-operad $\mathcal{O}$ further restricts the possibilities for the families $\set{\mathcal{F}_n(\mathcal{O})}_n$. In \cite{blumbergOperadicMultiplicationsEquivariant2015}, Blumberg and Hill define the notion of an \emph{indexing system} to capture the combinatorial properties that this sequence of families must satisfy. 

Denote the orbit category of $G$ by $\OrbitCategory$. We have a functor $\underline{\Set}_{fin}:\OrbitCategory^{op}\to \operatorname{Cat}$ where $\underline{\Set}_{fin}(G/H)= \Set_{fin}^H$ which Blumberg-Hill call the symmetric monoidal coefficient system of finite sets \cite[Definition 3.9]{blumbergOperadicMultiplicationsEquivariant2015}. 

\begin{definition}[\protect{\cite[Definition 3.22]{blumbergOperadicMultiplicationsEquivariant2015}}]
A $G$-indexing system  is a subfunctor $F$ of $\underline{\Set}_{fin}$ such that the following conditions hold:
\begin{enumerate}
    \item The subcategories $F(G/H)$ are closed under disjoint unions and cartesian products;
    \item If we have an inclusion of $H$-sets $X \hookrightarrow Y$ and $Y\in F(G/H)$, then $X\in F(G/H)$; and 
    \item If $X\in F(G/K)$ and $H/K\in F(G/K)$, then $H\times_K X\in F(G/H)$. This condition is referred to as self-induction.
\end{enumerate}
We will denote the category of $G$-indexing systems by $\CategoryOfIndexingSystems{G}$, where morphisms are inclusions of functors.
\end{definition}

Given a $H$-set $T$, we will denote the action homomorphism by $\phi_T:H\to \Sigma_T$ and the graph subgroup by $\Gamma_T \subseteq G\times \Sigma_T$. A $H$-set $T$ is \emph{admissible} for a $G$-operad $P$ if, after some ordering on the elements of $T\cong \vn[finset]$, the graph subgroup $\Gamma_T$ that corresponds to $T$ is such that $P(n)^{\Gamma_T}\neq \emptyset$. Phrased differently, a $H$-set $T$ is admissible for $P$ if $P$-algebras have norm maps of the form determined by $T$.
We can define a subfunctor $\AdmissibleSetsOf{P}$ of $\underline{\Set}_{fin}$ where $\AdmissibleSetsOf{P}(G/H)$ is the subcategory of $\Set^H_{fin}$ of all $H$-sets admissible for the operad $P$. 

In general, $\AdmissibleSetsOf{P}$ may not define an indexing system for a general $G$-operad $P$. Since we will be building operads that aren't necessarily $\mathbb{N}_\infty$-operads, let us clarify under what conditions $\AdmissibleSetsOf{P}$ gives an indexing system. 

Expanding on notation previously used, given a general $G$-operad $P$, let us denote the family of subgroups which fix points as \[\FixedSubgroupsOf[n]{P}:=\set[\bigg]{\Phi \leq G\times\Sigma_n\given P(n)^\Phi\neq \emptyset}.\] 
First, observe that we require $P$ to be $\Sigma$-free to ensure that every subgroup in $\FixedSubgroupsOf[n]{P}$ is a graph subgroup, and so is represented by an element of $\AdmissibleSetsOf{P}$. i.e., $\Sigma$-free ensures that the subfunctor $\AdmissibleSetsOf{P}$ encodes the same data as the sequence of subgroup families $\set{\FixedSubgroupsOf[n]{P}}_n$. 
\begin{lemma}
    If $P$ is $\Sigma$-free, then every element of $\FixedSubgroupsOf[n]{P}$ is represented by an element of $\AdmissibleSetsOf{P}$.
\end{lemma}
Looking at the proofs of \cite[Lemma 4.10, Lemma 4.11 and Lemma 4.12]{blumbergOperadicMultiplicationsEquivariant2015}, where Blumberg-Hill prove that $\AdmissibleSetsOf{P}$ is an indexing system when $P$ is an $\mathbb{N}_\infty$-operad, observe that statements about contractability are only needed to show certain fixed points spaces are non-empty. In particular, the only condition we need to ensure $\AdmissibleSetsOf{P}$ is an indexing system is the following.

\begin{restatable}{lemma}{sufficientforindexing}
    For a $G$-operad $P$, if for all $H\leq G$ we have $H\times \set{\id}\in \FixedSubgroupsOf[2]{P}$, then $\AdmissibleSetsOf{P}$ is an indexing system.
\end{restatable}

There have since been a few different equivalent formulations that encode the same data as an indexing system. The one we are most interested in this paper is the notion of a \emph{transfer system} (\cite{balchinNinftyoperadsAssociahedra2021,rubinDetectingSteinerLinear2021})

\begin{definition}
    Denote the subgroup lattice by $(\SubgroupLattice{G},\leq)$. A transfer system $\to$ of $G$ is a refinement of the subgroup lattice such that the following holds
    \begin{enumerate}
        \item For all $H\leq G$, then $H\to H$,
        \item If $K\to H$, then $K^g\to H^g$ for all $g\in G$,
        \item If $K\to H$ and $H\to L$, then $K\to L$,
        \item If $K\to H$ and $\vH[']\leq H$, then $K\cap \vH[']\to \vH[']$. 
    \end{enumerate}
    Transfer systems of $G$ then form a poset under inclusion which we will denote by $\TransferSystemsOf{G}$. We will often use a greek letter such as $\tau$ for a transfer system, and write $K\Transfer{\tau} H$ for the transfers it contains.
\end{definition}

The connection between transfer system and indexing systems starts by noticing that indexing systems are closed under subobjects and disjoint unions; and so are completely determined by what orbits $H/K$ it contains. In particular, we can define the following two functors:

\begin{gather*}
    \CoefficientSystemOf[uline]:\TransferSystemsOf{G}\to\CategoryOfIndexingSystems{G} \\
    \CoefficientSystemOf[uline]{\tau}(G/H) = \set[\bigg]{X\in \Set^H_{fin}\given X\cong \amalg_i H/K_i\text{ and } K_i \Transfer{\tau} H \text{ for all }i} \shortintertext{and, }
    \TransfersOf:\CategoryOfIndexingSystems{G}\to \TransferSystemsOf{G} \\ 
    \TransfersOf{I} = \set[\bigg]{K\leq H \given H/K\in I(G/H)}
\end{gather*}
The functors $\CoefficientSystemOf[uline]$ and $\TransfersOf$ are inverse functors of each other \cite[Theorem 3.7]{rubinDetectingSteinerLinear2021}, and establish the equivalence between indexing systems and transfer systems. For a $G$-operad $P$, we will abuse notation and use $\TransfersOf{P}$ also for the composition $\TransfersOf{\AdmissibleSetsOf{P}}$. 

\begin{definition}\label{definition: realize transfer system}
    A $G$-operad $P$ \emph{realizes} a transfer system $\tau\in \TransferSystemsOf{G}$ if it is $\Sigma$-free, for all $H\leq G$ we have $H\times \set{\id}\in \FixedSubgroupsOf[2]{P}$, and $\TransfersOf{P}=\tau$.
\end{definition}

\subsection{Coinduction and Suboperads}
As explained in the introduction, the idea behind our construction is to find suboperads of coinduced operads $\Coinduced[e,pow=G]P$ that have the correct fixed points. 
This idea has been considered previously. 
As communicated to the author by Mike Hill, one of the first attempts to prove the essential surjectivity of \cref{equation: Ninf homotopy to transfer} was to find candidate suboperads of the equivariant Barratt-Eccles operad $\BarrattEccles_G$, which is constructed by a coinduction operation. Let us quickly recall this operad: as explained in \cite[Definition 3.3, Lemma 3.3]{rubinCombinatorial$N_infty$Operads2021}, there is a functor \[E: \Operads(\Set^G)\to \Operads(\Spaces^G)\] such that for $P\in \Operads(\Set^G)$, the $G$-space $EP(n)$ is a universal space for $\FixedSubgroupsOf[n]{P}$. 
The Equivariant Barratt-Eccles operad $\BarrattEccles_G$ is then $E(\Coinduced[e,pow=G]\Assoc)$ where $\Assoc$ is the standard associative operad $\Assoc{n}=\Sigma_n$. The $G$-operad $\Coinduced[e,pow=G]\Assoc$ only has finitely many elements in each component, and so, it seems reasonable to try and find a suboperad of $\Coinduced[e,pow=G]\Assoc$ that realizes a specified transfer system before passing through to $E$. Unfortunately, Bonventre \cite[Example B.2.1.]{bonventreComparisonModelsEquivariant2017} gave a counterexample to this approach, and it appears this idea was abandoned soon after.

In order to motivate our construction, let us consider the little $k$-cubes operad $\LittleCubeOperad[k]$ and examine the fixed points of $\Coinduced[e,pow=G]\LittleCubeOperad[k]$. An element of the coinduced operad $x\in \Coinduced[e,pow=G]\LittleCubeOperad[k](n)$ can be interpreted as a map $x:G\times \vn[finset]\to \LittleCubeOperad[k](1)$ such that for any $g\in G$ and $i\neq j\in \vn[finset]$, the embeddings $x(g,i)$ and $x(g,j)$ have disjoint images. Let us go over two illustrative examples of what the fixed points of $\Coinduced[e,pow=G]\LittleCubeOperad[k](n)$ are for different groups $G$ and graph subgroups $\Gamma$.

\begin{example}
    Suppose $G=C_3$ with generator $\sigma$ and $\Gamma \leq C_3\times \Sigma_4$ is the graph subgroup generated by $(\sigma,(2,3,4))$. We can picture a general element $x\in \Coinduced[pow=G,e]\LittleCubeOperad[k](4)$ as an array where the rows are indexed by elements of $C_3$ and the columns indexed by the elements of $\underline{4}$, as in \cref{basic array picture}. 
    \begin{figure}[!ht]
        \begin{center}
        \begin{tikzpicture}
        \matrix[matrix of math nodes, left delimiter = {(},right delimiter = {)}, nodes = {minimum width=50px, minimum height = 20px}]{
            \node{x(e,1)};& \node{x(e,2)}; & \node{x(e,3)};& \node{x(e,4)};\\
            \node{x(\sigma,1)};& \node{x(\sigma,2)}; & \node{x(\sigma,3)};& \node{x(\sigma,4)};\\
            \node{x(\sigma^2,1)};& \node{x(\sigma^2,2)}; & \node{x(\sigma^2,3)};& \node{x(\sigma^2,4)};\\
        };
        \end{tikzpicture}
        \caption{The element $x$ pictured as an array of embeddings.}
        \label{basic array picture}
        \end{center}
    \end{figure}
    The group action $C_3$ then acts by permuting rows, while $\Sigma_4$ acts by permuting columns. If $x\in \Coinduced[pow=G,e]\LittleCubeOperad[k](4)^{\Gamma}$, then we must have $$\left((\sigma,(2,3,4))\cdot x\right)(g,i) = x(\vsigma[inv]g,(4,3,2)i) = x(g,i).$$ 
    In terms of our array picture, if we color components that are forced equal (see \cref{twisted column simple example}) they form ``twisted columns'' which reflects the structure of the corresponding admissible set $\SetFromGraph[\Gamma] = C_3/C_3\amalg C_3/e.$ Note that components of different colors can never be equal since embeddings in the same row must have disjoint image.
    \begin{figure}[!ht]
        \begin{center}
        \begin{tikzpicture}
        \matrix[matrix of math nodes, left delimiter = {(},right delimiter = {)}, nodes = {minimum width=50px, minimum height = 20px}]{
            \node[fill=orange!30]{x(e,1)};& \node[fill=purple!30]{x(e,2)}; & \node[fill=green!30]{x(e,3)};& \node[fill=blue!30]{x(e,4)};\\
            \node[fill=orange!30]{x(\sigma,1)};& \node[fill=blue!30]{x(\sigma,2)}; & \node[fill=purple!30]{x(\sigma,3)};& \node[fill=green!30]{x(\sigma,4)};\\
            \node[fill=orange!30]{x(\sigma^2,1)};& \node[fill=green!30]{x(\sigma^2,2)}; & \node[fill=blue!30]{x(\sigma^2,3)};& \node[fill=purple!30]{x(\sigma^2,4)};\\
        };
        \end{tikzpicture}
        \caption{An array picture of an element of $\Coinduced[pow=C_3,e]\LittleCubeOperad[k](4)^{\Gamma}$}\label{twisted column simple example}
        \end{center}
    \end{figure}
\end{example}

\begin{example}
    Consider the slightly more complex case where $G=C_2\inner{\tau}\times C_3\inner{\sigma}$ and $\Gamma$ is the graph subgroup of $G\times \Sigma_3$ generated by $((1,\sigma),(123))$. An element $x\in \Coinduced[e,pow=G]\LittleCubeOperad[k](3)^\Gamma$ can then be visualized as in \cref{more complicated twisted column example}.

    \begin{figure}[!ht]
        \begin{center}
        \begin{tikzpicture}
        \matrix[matrix of math nodes, left delimiter = {(},right delimiter = {)}, nodes = {minimum width=65px, minimum height = 20px}]{
            \node[fill=orange!30]{x((e,e),1)};& \node[fill=purple!30]{x((e,e),2)}; & \node[fill=green!30]{x((e,e),3)};\\
            \node[fill=green!30]{x((e,\sigma),1)};& \node[fill=orange!30]{x((e,\sigma),2)};&\node[fill=purple!30]{x((e,\sigma),3)};\\
            \node[fill=purple!30]{x((e,\sigma^2),1)};& \node[fill=green!30]{x((e,\sigma^2),2)}; & \node[fill=orange!30]{x((e,\sigma^2),3)};\\
            \node[fill=blue!30]{x(\tau,e),1)};& \node[fill=yellow!30]{x((\tau,e),2)}; & \node[fill=red!30]{x((\tau,e),3)};\\
            \node[fill=red!30]{x((\tau,\sigma),1)};& \node[fill=blue!30]{x((\tau,\sigma),2)};&\node[fill=yellow!30]{x((\tau,\sigma),3)};\\
            \node[fill=yellow!30]{x((\tau,\sigma^2),1)};& \node[fill=red!30]{x((\tau,\sigma^2),2)}; & \node[fill=blue!30]{x((\tau,\sigma^2),3)};\\
        };
        \end{tikzpicture}
        \caption{An array picture of an element of $\Coinduced[pow=C_6,e]\LittleCubeOperad[k](4)^{\Gamma}$}
        \label{more complicated twisted column example}
        \end{center}
    \end{figure}
    Unlike in the previous example, the twisted columns of components that are forced equal don't span the entire array. This is because the graph subgroup $\Gamma$ is such that $H_\Gamma = C_3\inner{\sigma} \subseteq G$, and our twisted columns only span the indices corresponding to a coset of $H_\Gamma$. Again, components with the same color are forced to be equal, and those with colors that appear in the same row are forced to not be equal.
\end{example}

As these examples illustrate, fixed points of $\Coinduced[e,pow=G]\LittleCubeOperad[k](n)$ are composed of twisted columns, whose form is determined by the graph subgroup that fixes it. Our construction will be based on controlling what kind of twisted columns are possible after operadic composition. One important observation in this regard is the following: given embeddings $x_1,x_2,y_1,y_2\in \LittleCubeOperad[k](1)$, if $x_1$ and $x_2$ have disjoint image, then $x_1y_1$ and $x_2y_2$ will also have disjoint image. In the context of array picture, this means that once two elements have disjoint image in an array, there is no further composition that could make any of their composites equal and part of a twisted column.

This behavior turns out to be general, and we will abstract this property in the next section. It will be the key property that will allow us to control which twisted columns can appear.

%% file: content/intersectionmonoids.tex
\section{Intersection Monoids and Operads}\label{section: intersection monoids and operads}
In this section, we will write $\vC[scr]$ for the category $\Set$ or $\Spaces$.
We define the following to abstract the main properties we need from the monoid $\LittleCubeOperad[k](1)$. 

\begin{definition}\label{definition: intersection monoids}
    An intersection monoid $(M,\Intersect_M)$ is a monoid $M$ in $\vC[scr]$ with a symmetric, reflexive relation $\Intersect_M$ on $M$ such that
    \begin{enumerate}
        \item\label[condition]{condition: right invariance} for all $\vx[1],\vx[2],\vy[1],\vy[2]\in M$, if $\vx[1]\vy[1]\Intersect_M \vx[2]\vy[2]$ then $\vx[1]\Intersect_M \vx[2]$, and
        \item\label[condition]{condition: left partial invariance} for all $x,\vy[1],\vy[2]\in M$, if $x\vy[1]\Intersect_M x\vy[2]$ then $\vy[1]\Intersect_M \vy[2]$.
    \end{enumerate}
    We will denote the complement relation to $\Intersect_M$ by $\Disjoint_M$. If $x\Intersect_M y$, then we say $x$ and $y$ intersect, and if $x\Disjoint_M y$, we say they are disjoint. An intersection monoid is \emph{trivial} if for all $\vx[1],\vx[2]\in M$, we have $\vx[1]\Intersect_M\vx[2]$. It may be useful to consider the above conditions in terms of the disjoint relation $\Disjoint_M$:
    \begin{enumerate}[label = ($\arabic*^\prime$)]
        \item for all $\vx[1],\vx[2],\vy[1],\vy[2]\in M$, if $\vx[1]\Disjoint_M \vx[2]$, then $\vx[1]\vy[1]\Disjoint_M \vx[2]\vy[2]$, and 
        \item for all $x,\vy[1],\vy[2]\in M$, if $\vy[1]\Disjoint_M \vy[2]$, then $\vx\vy[1]\Disjoint_M \vx\vy[2]$.
    \end{enumerate}

    A morphism of intersection monoids $f:(M,\Intersect_M)\to (N,\Intersect_N)$ is a monoid homomorphism $f:M\to N$ such that if $x\Disjoint_M y$ then $f(x)\Disjoint_N f(y)$. We then have a category of intersection monoids in $\vC[scr]$ which we denote by $\CategoryOfIntersectionMonoids(\vC[scr])$.
\end{definition}

One useful observation from this definition is the following.

\begin{lemma}\label{lemma: nontrivial have infinite disjoint elements}
    If $M$ is a non-trivial intersection monoid, then for any $n\in \mathbb{N}$, there exists a family of elements $\set{\vx[i]}_{i\in \vn[finset]}$ which are pairwise disjoint: $\vx[i]\Disjoint \vx[j]$ for all $i\neq j$.
\end{lemma}

\begin{proof}
    We prove this via induction. The base case of $n=2$ holds by the definition of $M$ being non-trivial. Suppose now we have a family $\set{\vx[i]}_{i\in \vn[finset]}$ of pairwise disjoint elements where $n\geq 2$. We construct a new family $\set{\vy[i]}_{i\in\underline{n+1}}$ where \[\vy[i]=\begin{cases*}
        \vx[i] & if $i<n$ \\ \vx[n]\vx[1] &if $i=n$ \\ \vx[n]\vx[2] &if $i=n+1$. 
    \end{cases*}\] This is pairwise disjoint by the axioms of an intersection monoid.
\end{proof}

\begin{example}\label{example: basic examples of intersection monoids}
    \begin{enumerate}
        \item The monoid of unary elements of the little $k$-cubes operad $\LittleCubeOperad_k(1)$ is an intersection monoid where for $x,y\in \LittleCubeOperad_k(1)$, the intersection relation is given by $x\wedge y$ if and only if $x(\text{int}(I^k))\cap y(\text{int}(I^k))$ is non-empty.
        \item The monoid of unary elements of the linear isometries operad $\LinearIsometriesOperad(1)$ is an intersection monoid where for $x,y\in \LinearIsometriesOperad(1)$, the intersection relation is given by $x\Intersect y$ if and only if $\Image{x}^{\perp}\cap \Image{y}\neq \set{0}$. Equivalently, $x\Disjoint y$ if and only if $\Image{x}\perp \Image{y}$.
        \item Consider the free monoid generated by two letters $\DyadicIntervalMonoid=F(\set{a,b})$. For a general word \[w=\ell_1\ell_2\dots\ell_n\] where $\ell_k\in \set{a,b}$, we will write the length by $\Length{w}=n$ and the $k$-th letter by $w^k:=\ell_k$. We can put an intersection relation on $\DyadicIntervalMonoid$ where given two words $w_1,w_2\in \DyadicIntervalMonoid$, we set $w_1\Disjoint w_2$ if and only if there exists a $k\leq \min(\Length{\vw[1]},\Length{\vw[2]})$ such that $w_1^k\neq w_2^k$. We call $\DyadicIntervalMonoid$ the \emph{dyadic interval monoid}. We will explain this name in \cref{section: realizing ninf}.
    \end{enumerate}
\end{example}

\begin{remark}
    An intersection relation on a monoid $M$ is related to divisibility properties of that monoid. Given any monoid $M$, we can attempt to put an intersection relation $\Intersect$ on $M$ where for $x,y\in M$, we set $x\Intersect y$ if and only if there exists $\vx['],\vy[']$ such that $\vx\vx[']=\vy\vy[']$. This relation is reflexive and symmetric and satisfies \cref{condition: right invariance} of \cref{definition: intersection monoids}. If we further assume that $M$ is a left domain (if $xy=\vx\vy[']$ in $M$, then $\vy=\vy[']$), then $M$ also satisfies \cref{condition: left partial invariance} and so is an intersection monoid under this relation. If there was any other intersection monoid structure $\Intersect[~]$ on $M$, then if $x\Disjoint[~]y$, we must also have $x\Disjoint y$. This can be seen as if $x\Intersect y$, then there exists $\vx['],\vy[']\in M$ such that $\vx\vx[']=\vy\vy[']$, and so $\vx\vx[']\Intersect[~]\vy\vy[']$ by reflexivity. This, however, would contradict $x\Disjoint[~]y$ and \cref{condition: right invariance} of \cref{definition: intersection monoids}. As a consequence, for a left domain monoid $M$, we see that the intersection monoid $(M,\Intersect)$ is terminal for the subcategory of $\CategoryOfIntersectionMonoids$ of intersection monoids with underlying monoid $M$ and inclusions.
\end{remark}

\begin{definition}
    Given a monoid $M$, we build a $\vC[scr,pow=G]$-valued operad $\OConstr[pow=G]{M}$ as follows. The $n$-th space is given by \[\OConstr[pow=G]{M}(n) := \Map(G\times \vn[finset],M)\] where the $G\times \Sigma_n$-action is given by \[\left((g,\sigma)\cdot x\right)(k,i) = x(\vg[inv]k,\vsigma[inv]{i}).\] Here we take $\OConstr[pow=G]{M}(0)=\ast$.

    We will define the operadic structure of $\OConstr[pow=G]{M}$ in terms of $\circ_i$-composition. Given $n,m\in \mathbb{N}$, and $i\in \vn[finset]$, we have two functions defined as follows. 
    The \emph{$i$-collapse function} $\Combine[i,pow={n,m}]\colon\underline{n+m-1}\to \vn[finset]$ by \begin{align*}
        \Combine[i,pow={n,m}](k) :=\begin{cases*}
            k & if $k<i$ \\ i & if $i \leq k \leq i+m - 1$ \\ k-m+1 & otherwise. 
        \end{cases*}
    \end{align*}
    As well as the \emph{$i$-th shift function} $\Shift[i,pow=m]\colon \set{i,i+1,\dots,i+m-1}\to \vm[finset]$ defined by $\Shift[i,pow=m](k):= k-i+1$.  Then for $x\in \OConstr[pow=G]{M}(n)$, and $y\in\OConstr[pow=G]{M}{m}$, the $\circ_i$-composition is given by
    \begin{align}\label{equation: partial composition definition}
        (x\circ_i y)(g,k) = \begin{cases*}
            x(g,\Combine[i,pow={n,m}](k))y(g,\Shift[i,pow= m](k)) & if $i\leq k \leq i+m-1$ \\ 
            x(g, \Combine[i,pow={n,m}](k)) & otherwise.
        \end{cases*}
    \end{align}
    We will use the convention that \[y(g,\Shift[oline,i](k)) = \begin{cases*}
        y(g,\Shift[i,pow=m](k)) & if $i\leq k \leq i+m-1$ \\ 
        1 & otherwise,
    \end{cases*}\] where $1$ is the unit of the monoid $M$.
    We can then express $\circ_i$-composition as 
    \begin{align}\label{equation: compact partial composition definition}
        (x\circ_i y)(g,k) = x(g,\Combine[i](k))y(g,\Shift[oline,i](k))
    \end{align}
    When $y=\ast\in \OConstr[pow=G]{M}(0)$, composition is given by \[ (x\circ_i y)(g,k) = \begin{cases*}
            x(g,k) & if $k < i$ \\ 
            x(g, k+1) & otherwise.
        \end{cases*}\]
\end{definition}

Let us verify that $\OConstr[pow=G]{M}$ is a $G$-operad as defined.

\begin{lemma}\label{lemma:O_construction_is_an_operad}
    Let $M$ be an intersection monoid. The $G$-symmetric sequence $\OConstr[pow=G]{M}$ is a well-defined operad in $\vC[cal,pow=G]$.
\end{lemma}

\begin{proof}
    The unit element $u\in \OConstr[pow=G]{M}$ is simply given by $u(g,i)=1$ for all $(g,i)\in G\times\vn[finset]$. Associativity of the composition maps follows from associativity of the monoid product of $M$.
    Let us verify the equivariance condition, \cref{eq:operad_equivariance}.
    Let $x\in \OConstr[pow=G]{M}(n)$, $y\in\OConstr[pow=G]{M}(m)$, $\vsigma\in\vSigma[n]$, and $\vtau\in\vSigma[m]$. Observe that we have the following equations:
    \begin{gather*}
        \vsigma[inv]\Combine[\vsigma{i}]{k} = \Combine[i]{(\vsigma[inv]\circ_{\vsigma{i}}\id)(k)}  = \Combine[i]{(\vsigma[inv]\circ_{\vsigma{i}}\vtau[inv])(k)} \shortintertext{and,}
        \vtau[inv]\Shift[\vsigma{i}](k) = \Shift[\vsigma{i}]((\id\circ_{\sigma(i)}\vtau[inv])(k)) = \Shift[i]((\vsigma[inv]\circ_{\sigma(i)}\vtau[inv])(k)).
    \end{gather*}
    Moreover, not that $(\sigma\circ_i\tau)^{-1}= (\vsigma[inv]\circ_{\vsigma{i}}\vtau[inv])$. 
    Hence, we obtain that 
    \begin{align*}
        \left(\sigma\cdot \vx\right)\circ_{\vsigma{i}} (\tau \cdot \vy)(g,k) &= [\sigma\cdot\vx](g,\Combine[\vsigma{i}]{k})[\tau\cdot\vy](g,\Shift[\vsigma{i},oline](k)) \\
        &=\vx(g,\vsigma[inv]\Combine[\vsigma{i}]{k})\vy(g,\vtau[inv]\Shift[\vsigma{i},oline](k)) \\ 
        &= \vx(g,\Combine[i]{(\vsigma[inv]\circ_{\vsigma{i}}\vtau[inv])(k)})\vy(g,\Shift[i,oline]((\vsigma[inv]\circ_{\sigma(i)}\vtau[inv])(k)))  \\ 
        &= \left((\vsigma\circ_{i}\vtau)\cdot (\vx\circ_i \vy)\right)(g,k)
    \end{align*}
    as required.
\end{proof}

\begin{definition}
    If $M$ is an intersection monoid, we say an element $x\in \OConstr[pow=G](M)$ is \emph{$\Sigma$-disjoint} if for all $g\in G$, and $i\neq j\in \vn[finset]$ we have $\vx{g,i}\Disjoint \vx{g,j}.$
    We say $x$ has \emph{strict columns} if for any $i\in\vn[finset]$, and $\vg[1],\vg[2]\in G$ we have either $\vx{\vg[1],i}=\vx{\vg[2],i}$ or $\vx{\vg[1],i}\Disjoint\vx{\vg[2],i}$
\end{definition}

The following lemma is a straight forward consequence of the axioms of an intersection monoid and how the composition is defined.

\begin{lemma}\label{lemma:sigma disjoint and strict columns closed under composition}
    Let $M$ be an intersection monoid, then $\Sigma$-disjoint elements and strict column elements are closed under composition. That is, if $x\in \OConstr[pow=G](M)(n)$, $y\in \OConstr[pow=G](M)(m)$ are $\Sigma$-disjoint or have strict columns, then the composition $\vx\circ_i \vy$ is $\Sigma$-disjoint or has strict columns respectively.
\end{lemma}

\begin{definition}
    We will denote the sub-$G$-symmetric sequence of $\OConstr[pow=G]{M}$ given by all $\Sigma$-disjoint elements that have strict columns by $\Realize[pow=G]{M}$. The sequence $\Realize[pow=G]{M}$ contains the unit and $G$-invariant, and so \cref{lemma:sigma disjoint and strict columns closed under composition} tells us that this is a suboperad. \Cref{lemma: nontrivial have infinite disjoint elements} also gives us that $\Realize[pow=G]{M}{n}$ is not empty when $M$ is non-trivial. We will call $\Realize[pow=G]{M}$ the \emph{$G$-incomplete realization by $M$} operad.
\end{definition}

Since a morphism of intersection monoids $f:M\to N$ preserves disjointness, we see that $\Realize[pow=G]$ is functorial on intersection monoids.

\begin{example}
    When $G= \set{1}$, the functor $\Realize[pow=G]$ recovers the following operads from \cref{example: basic examples of intersection monoids} from their unary components and intersection relations.
    \begin{align*}
        \Realize[pow=\set{1}]{\LittleCubeOperad_k(1)} &= \LittleCubeOperad_k \\
        \Realize[pow=\set{1}]{\LinearIsometriesOperad(1)} &= \LinearIsometriesOperad.
    \end{align*}
\end{example}

\begin{lemma}\label{lemma: equivalence of intersection monoids and intersection operads}
    The functor $\Realize[pow=\set{1}]:\CategoryOfIntersectionMonoids(\vC[scr])\to \Operads(\vC[scr])$ is fully faithful. 
\end{lemma}
\begin{proof}
    We have $\Realize[pow=\set{1}]{M}(1)=M$ as monoids, and $\Realize[pow=\set{1}]{M}(2) \subseteq M^2$ is exactly the disjoint relation $\Disjoint$ on $M$. As such, the faithfulness of $\Realize[pow=\set{1}]$ is straightforward.
    
    To show that it is full, suppose we have a morphism of operads \[f:\Realize[pow=\set{1}]{M}\to \Realize[pow=\set{1}]{N}.\] 
    For each level $n$, we have \[f(n): \Realize[pow=\set{1}]{M}(n) \subseteq M^n \to \Realize[pow=\set{1}]{N} \subseteq N^n.\]
    Write $f(n)_k$ for the $k$-th projection onto the $k$-th factor of $N^n$, so $f(n)= (f(n)_1,f(n)_2,\dots,f(n)_n)$. Composition with $\ast\in\Realize[pow=\set{1}]{M}(0)$ in each component except for the $k$-th position yields the following commutative diagram (where $\widehat{(-)}$ means we omit that operation).
    \begin{center}
        \begin{tikzpicture}[commutative diagrams/every diagram]
            \matrix[matrix of math nodes, name=m, commutative diagrams/every cell,column sep=4cm] {
            \mathcal{R}^{\{1\}}(M)(n) & \mathcal{R}^{\{1\}}(N)(n) \\
            \mathcal{R}^{\{1\}}(M)(1) & \mathcal{R}^{\{1\}}(N)(1) \\};

            \path[commutative diagrams/.cd, every arrow, every label]
            (m-1-1) edge node {$(f(n)_1,f(n)_2,\dots,f(n)_n)$} (m-1-2)
            edge node[swap] {$-\circ_1\ast \circ_2\ast \dots \widehat{\circ_k\ast} \dots \circ_n \ast $}(m-2-1)
            (m-2-1) edge node {$f(n)_k(\ast,\ast,\dots,-,\dots,\ast)$} (m-2-2)
            (m-1-2) edge node {$-\circ_1\ast \circ_2\ast \dots \widehat{\circ_k\ast} \dots \circ_n \ast $} (m-2-2);
        \end{tikzpicture}
    \end{center}
    Since $f$ is an operad morphism, this implies $f(n)_k(\ast,\ast,\dots,-,\dots,\ast)=f(1)$ for each $k$. It follows that $f(n)=f(1)^n$ and we deduce that the functor is full.
\end{proof}

\begin{definition}
    An \emph{intersection operad} $P$ is an operad in the \emph{image} of the functor \[\Realize[pow=\set{1}]: \CategoryOfIntersectionMonoids(\vC[scr])\to \Operads(\vC[scr]).\] We will denote the image of $\Realize[pow=\set{1}]$ by $\CategoryOfIntersectionOperads(\vC[scr])$.
\end{definition}

\Cref{lemma: equivalence of intersection monoids and intersection operads} tells us that $\CategoryOfIntersectionMonoids(\vC[scr])$ and $\CategoryOfIntersectionOperads(\vC[scr])$ are isomorphic categories. Given an intersection operad $P$, we will interpret $P(1)$ as the intersection monoid that underlies it. Observe that we have $\Realize[pow=G]{P(1)}$ as a suboperad of $\Coinduced[pow=G,e]P$. This leads to our choice in the following definition.

\begin{definition}
    Let $P$ be an intersection operad. The \emph{$G$-incomplete coinduction} $\Coinduced[pow=G,e,uline]{P}$ of $P$ is the operad $\Realize[pow=G]{P(1)}$.
\end{definition}

%% file: content/incompleteindexingoperad.tex
\section{Incomplete Indexing Operad}\label{section: indexing operad}

In this section we will define a useful operad that we will leverage in the next section. The purpose of this operad is to encode which components $x(g,i)$ of an element $x\in\Realize[pow=G]{M}$ are intersecting. We start with an operad that can encode any intersection between components:

\begin{definition}\label{definition:incomplete_indexing_operad}
    The \emph{complete indexing operad} $\IncompleteIndexOperad[pow=G]$ is a $\Poset^G$-valued operad where each level $\IncompleteIndexOperad[pow=G](n)$ the set of simple undirected graphs on the vertex set $G\times \vn[finset]$ such that for any $g\in G$ and $i,j\in\vn[finset]$, there is no edge $\edge{(g,i)}{(g,j)}$. This forms a poset under graph inclusion.

    There is a natural $G\times \Sigma_n$-action on $\IncompleteIndexOperad[pow=G](n)$, where for $\vk[graph]\in \IncompleteIndexOperad[pow=G](n)$, $g\in G$ and $\sigma\in\Sigma_n$, the graph $(g,\sigma)\cdot \vk[graph]$ is the graph where the edge $\edge{(h_1,i_1)}{(h_2,i_2)}$ is in $(g,\sigma)\cdot \vk[graph]$ if $\vk[graph]$ has the edge $\edge{(\vg[inv]h_1,\vsigma[inv]i_1)}{(\vg[inv]h_2,\vsigma[inv]i_2)}$.

    We get an operad structure as follows: given graphs $\vg[graph]\in \IncompleteIndexOperad[pow=G]{n}$ and $\vk[graph]\in \IncompleteIndexOperad[pow=G]{m}$, we define the $\circ_i$-composition $\vg[graph]\circ_i\vk[graph]$ as the graph where the edge $\edge{(g_1,j_1)}{(g_2,j_2)}$ is in $\vg[graph]\circ_i\vk[graph]$ if 
    \begin{enumerate}
        \item $\edge{(g_1,\Combine[i]{j_1})}{(g_2,\Combine[i]{j_2})}$ is an edge of $\vg[graph]$; and
        \item if $\Combine[i]{j_1}=\Combine[i]{j_2}=i$, then $\edge{(g_1,\Shift[i]{j_1})}{(g_2,\Shift[i]{j_2})}$ is an edge of $\vk[graph]$.
    \end{enumerate}
    The unit of the operad $\IncompleteIndexOperad[pow=G]$ is the complete graph on the vertex set $G\times \underline{1}$. The verification that $\IncompleteIndexOperad[pow=G]$ is an operad is straight forward, but tedious. The interested reader can find the details in \cref{appendix:indexing_is_an_operad}.
\end{definition}

We will want suboperads of $\IncompleteIndexOperad[pow=G]$ that encode only intersections possible for a specified transfer system $\tau\in\TransferSystemsOf{G}$. To construct this, let us introduce some notation and terminology. Let $S \subseteq G$, and suppose we have a function of the form $\alpha: S\to \vn[finset]$. We will denote the graph of this function by $\GraphOf{\alpha}$. i.e., \[\GraphOf{\alpha}= \set[\bigg]{(s,\valpha{s})\given s\in S}.\]
\begin{definition}
    Suppose we have \emph{subsets} $T \subseteq S \subseteq G$. A graph $\vg[graph]\in \IncompleteIndexOperad[pow=G](n)$ has a \emph{$S/T$-complete subgraph at position $i\in\vn[finset]$} if there exists a function $\alpha\colon S\to \vn[finset]$ such that $\valpha[inv]{i}=T$ and the complete graph on the vertex set $\vGamma{\alpha}$, written $\CompleteGraph{\vGamma{\alpha}}$, is a subgraph of $\vg[graph]$. We call $\alpha\colon S\to \vn[finset]$ the \emph{underlying function} for the $S/T$-complete subgraph at position $i$.
\end{definition}

\begin{lemma}\label{lemma: twists of incomplete graph}
    Let $\vg[graph]\in \IncompleteIndexOperad[pow=G]{n}$, $\vk[graph]\in \IncompleteIndexOperad[pow=G]{m}$, $T \subseteq S \subseteq G$ be subsets and write \[A=\Combine[i,pow={n,m},spar,inv]{i}=\set{i,i+1,\dots,i+m-1}.\] Suppose $\vg[graph]\circ_i \vk[graph]$ has a $S/T$-complete subgraph at position $j\in \underline{n+m-1}$. Then \begin{enumerate}
        \item if $j\notin A$, then $\vg[graph]$ has a $S/T$-complete subgraph at position $\Combine[i,pow={n,m}]{j}$, or 
        \item if $j\in A$, then there exists a subset $R \subseteq G$, with $T \subseteq R \subseteq S$ such that $\vg[graph]$ has a $S/R$-complete subgraph at position $i$ and $\vk[graph]$ has a $R/T$-complete subgraph at position $j-i+1$.
    \end{enumerate}
\end{lemma}

\begin{proof}
    Suppose $\alpha:S\to \underline{m+n-1}$ is the function underlying the $S/T$-complete subgraph at position $j$ for $\vg[graph]\circ_i \vk[graph]$. Define $\beta: S\to \vn[finset]$ as the composite $\beta:= \Combine[i,pow={n,m}]\circ \alpha$.

    \emph{Claim 1: if $\CompleteGraph{\GraphOf{\alpha}}$ is a subgraph of $\vg[graph]\circ_i\vk[graph]$, then $\CompleteGraph{\GraphOf{\beta}}$ is a subgraph of $\vg[graph]$.} 
    This follows as suppose $\CompleteGraph{\GraphOf{\beta}}$ wasn't a subgraph of $\vg[graph]$. Then there is some edge $\edge{(\vg_1,\vbeta{\vg_1})}{(\vg_2,\vbeta{\vg_2})}$ not in $\vg[graph]$. From the definition of the operad composition, the graph $\vg[graph]\circ_i\vk[graph]$ must not have the edge $\edge{(g_1,\valpha{g_1})}{(g_2,\valpha{g_2})}$ which contradicts $\CompleteGraph{\GraphOf{\alpha}}$ being a subgraph of $\vg[graph]\circ_i\vk[graph]$. This justifies the claim.

    If $\Image(\valpha)\cap A\neq\emptyset$, let us set $R:=\vbeta[inv]{i}$ and define a further map $\gamma:R\to \vm[finset]$ by the composite $\gamma=\Shift[i]\circ\restr{\alpha}{R}$.
    
    \emph{Claim 2: if $\CompleteGraph{\GraphOf{\alpha}}$ is a subgraph of $\vg[graph]\circ_i\vk[graph]$, then $\CompleteGraph{\GraphOf{\gamma}}$ is a subgraph of $\vk[graph]$.}
    Suppose $\CompleteGraph{\GraphOf{\gamma}}$ wasn't a subgraph of $\vk[graph]$, then there is some edge $\edge{(\vg_1,\vgamma{\vg_1})}{(\vg_2,\vgamma{\vg_2})}$ not in $\vk[graph]$. Again, by the operad composition, this would imply $\edge{(\vg_1,\valpha{\vg_1})}{(\vg_2,\valpha{\vg_2})}$ not in $\CompleteGraph{\GraphOf{\alpha}}$ which is a contradiction. Hence, we have proven the claim.

    The lemma follows from these two claims. If $j\notin A$, then $\vbeta[inv]{\Combine[i,pow={n,m}](j)}=\valpha[inv]{j}=T$ and so $\vbeta$ underlies a $S/T$-complete subgraph at position $\Combine[i,pow={n,m}](j)$. If $j\in A$, then similarly, $\vbeta$ underlies a $S/R$-complete subgraph at position $\Combine[i,pow={n,m}](j)=i$, and $\vgamma$ underlies a $R/T$-complete subgraph at $\Shift[i](j)=j-i+1$.
\end{proof}

\begin{definition}\label{definition: twist maps}
	A \emph{twist map for $G\times \vn[finset]$} is a map $\alpha\colon S\to \vn[finset]$ where $S \subseteq G$, and there exists a $\vg[0]\in G$ such that $S\vg[0]$ is a subgroup of $G$. We denote the set of all twists maps for $G\times \vn[finset]$ by $\TwistMaps{G,\vn[finset]}$. 
    A twist map $\alpha\colon S\to \vn[finset]$ \emph{structures a transfer $K\to H$ at position $i$} if there exists a $\vg[0]\in G$, and $i\in \Image{\valpha}$ such that $S\vg[0]=H$ and $K \subseteq \valpha[inv]{i}\vg[0]$ is a maximal inclusion among subgroups. i.e., if we have another subgroup $\vK[']$ where $K \subseteq \vK['] \subseteq \valpha[inv]{i}\vg[0]$, then $K=\vK[']$. 
\end{definition}

The set $\TwistMaps{G,\vn[finset]}$ has a $G\times\Sigma_n$-action where if $(\alpha\colon S\to \vn[finset])\in \TwistMaps{G,\vn[finset]}$, then $(g,\sigma)\cdot \alpha\colon\vg S\to \vn[finset]$ is the map given by $\left((g,\sigma)\cdot \alpha\right)(h)=\vsigma\alpha(\vg[inv]h)$. Note that if $S\vg[0]=H$, then $gS\vg[0]\vg[inv] = \prescript{g}{}{H}$. 
Hence, if $\vbeta= (g,\sigma)\cdot\alpha$, and $\alpha$ structures a transfer $K\to H$ at position $i$, 
then we have $\Domain{\vbeta}\vg[0]\vg[inv] = \vg\vS\vg[0]\vg[inv] = \prescript{g}{}{H}$, and $\vbeta[inv]{\vsigma{i}}\vg[0]\vg[inv]=g\valpha[inv]{i}\vg[0]\vg[inv] \subseteq  \prescript{g}{}{K}$ which is maximal. 
Hence, $\vbeta$ structures the transfer $\prescript{g}{}{K}\to \prescript{g}{}{H}$ at position $\sigma(i)$.
This observation implies the following is well-defined.

\begin{definition}
Let $\tau\in \TransferSystemsOf{G}$, we define the following $G\times\Sigma_n$-subsets \[\TwistMaps[pow=\tau]{\vn[finset],G}:= \set[\bigg]{\alpha \in \TwistMaps{\vn[finset],G}\given \text{If }\alpha\text{ structures }K\to H,\text{ then }K\Transfer{\tau} H}.\]
\end{definition}

\begin{definition}
    We will say that $\vg[graph]\in \IncompleteIndexOperad[pow=G]{n}$ \emph{supports a transfer $K\to H$ at position $i$} if $\vg[graph]$ has a $S/T$-complete subgraph at position $i$ where the underlying map $\alpha:S\to \vn[finset]$ is a twist map that structures a transfer $K\to H$ at position $i$. We say the function $\alpha:S\to \vn[finset]$ underlying the $S/T$-complete subgraph \emph{exhibits} the transfer $K\to H$.
\end{definition}

Explicitly, this says that we have $(\alpha:S\to \vn[finset])\in \TwistMaps(G,\vn[finset])$ and $g_0\in G$ where $\CompleteGraph{\alpha} \subseteq \vg[graph]$, $H=S\vg[0]$, and $K \subseteq \valpha[inv]{i}\vg[0]$ maximally. 

\begin{definition}
    For a transfer system $\tau\in\TransferSystemsOf{G}$, we define the \emph{$\tau$-incomplete indexing $G$-operad} $\IncompleteIndexOperad[pow=\tau]$ by:
\begin{align*}
    \SwapAboveDisplaySkip
    \IncompleteIndexOperad[pow=\tau](n) := \set[\bigg]{\vg[graph]\in \IncompleteIndexOperad[pow=G](n)\given \text{If }\vg[graph]\text{ supports a transfer }K\to H, \text{ then } K\Transfer{\tau} H}.
\end{align*}
\end{definition}

\begin{lemma}
    For a transfer system $\tau\in\TransferSystemsOf{G}$, the $G$-symmetric sequence $\IncompleteIndexOperad[pow=\tau]$ is a $G$-operad.
\end{lemma}
\begin{proof}
    First note that $\IncompleteIndexOperad[pow=\tau](n)$ is a $G\times\Sigma_n$-subset of $\IncompleteIndexOperad[pow=G](n)$, and the unit $\CompleteGraph{G\times\underline{1}}\in \IncompleteIndexOperad[pow=\tau]{1}$. Therefore, we only need to check that $\IncompleteIndexOperad[pow=G]$ is closed under composition. 
    
    Let $\vg[graph]\in \IncompleteIndexOperad[pow=\tau]{n}$ and $\vk[graph]\in \IncompleteIndexOperad[pow=\tau]{m}$, and set $\vh[graph]=\vg[graph]\circ_i\vk[graph]$. Suppose $\vh[graph]$ supports a transfer $K\to H$ which is exhibited by the function $\valpha:S\to \underline{m+n-1}$ which underlies a $S/T$-complete subgraph at position $j\in\underline{m+n-1}$. We need to show that $K\Transfer{\tau} H$. 

    Suppose $j\notin \Combine[i,inv]{i}$. In this case, \cref{lemma: twists of incomplete graph} implies that $\vg[graph]$ supports a $K\to H$ transfer, and so $K\Transfer{\tau}H$ by definition.

    Let us then consider the alternate case when $j\in \Combine[i,inv]{i}$. In this case, let $g_0\in G$ be the element such that $Sg_0=H$ and $K \subseteq \valpha[inv]{i}g_0$ maximally.
    Define the functions $\beta:S\to \vn[finset]$, and $\gamma:R\to \vm[finset]$ where $R=\vbeta[inv]{i}$ as in the proof of \cref{lemma: twists of incomplete graph}. 
    Here, $\beta$ underlies a $S/R$-complete subgraph for $\vg[graph]$ at position $i$, and $\gamma$ a $R/T$-complete subgraph for $\vk[graph]$ at position $j-i+1$. 
    Let $L$ be a maximal subgroup such that $K \leq L \subseteq R\vg_0$, and observe this implies $\beta$ exhibits the transfer $L\to H$. Since $\vg[graph]\in \IncompleteIndexOperad[pow=\tau]{n}$, by definition, we have $L\Transfer{\tau} H$. 
    
    Since $\gamma$ underlies a $R/T$-complete subgraph of $\vk[graph]$ at position $j-i+1$, it follows \[K \subseteq T\vg[0] = \vgamma[inv]{j-i+1}\vg[0]\subseteq R\vg[0].\] We also have $K\subseteq L \subseteq R\vg[0]$, and so $K \subseteq T\vg[0]\cap L \subseteq L \subset R\vg[0]$. 
    If we set $\vR[']= L\vg[0,inv]$ and $\vgamma[']=\restr{\gamma}{\vR[']}$, then \[K \subseteq T\vg[0]\cap \vR[']g_0 =\vgamma[',spar,inv](j-i+1)\vg[0]\subseteq \vR[']g_0 \subset R\vg[0].\] Moreover, $K \subseteq T\vg[0]\cap \vR[']g_0$ maximally, since otherwise this would contradict $K \subseteq T\vg[0]$ maximally. 
    Hence, the restricted map $\vgamma[']$ underlies a $\vR[']/(\vT\cap\vR['])$-complete subgraph of $\vk[graph]$ at position $j-i+1$, and $\vgamma[']$ structures a transfer $K\to L$. In other words, $\vgamma[']$ exhibits a transfer $K\to L$ for $\vk[graph]$, and so by definition $K\Transfer{\tau} L$. 
    
    Since both $K\Transfer{\tau} L$ and $L\Transfer{\tau} H$, we deduce that $K\Transfer{\tau} H$ since transfer systems are closed under compositions.
\end{proof}

\begin{warning}\label{warning: Bonventre graph version}
It is not the case that $\TransfersOf{\IncompleteIndexOperad[pow=\tau]}=\tau$. As an example, let us consider the following case. Let $G= C_4\inner{\sigma}$ and consider the transfer system given by \[\tau=\set[\bigg]{e\to C_2}.\]
Consider the following graph $\vg[graph]\in \IncompleteIndexOperad[pow=G](2)$ given by:
\begin{center}
\begin{tikzpicture}[commutative diagrams/every diagram]
    \matrix[matrix of math nodes, name=m, commutative diagrams/every cell] {
    (e,1) & (e,2) \\
    (\sigma^2,1) & (\sigma^2,2) \\
    (\sigma,1) & (\sigma,2) \\
    (\sigma^3,1) & (\sigma^3,2) \\
    };

    \path[commutative diagrams/.cd, every label]
    (m-1-1) edge (m-2-2)
    (m-2-1) edge (m-1-2)
    (m-3-1) edge (m-4-2)
    (m-4-1) edge (m-3-2);
\end{tikzpicture}
\end{center}
The graph $\vg[graph]$ only supports transfers of the form $e\to C_2$, and so $\vg[graph]\in \IncompleteIndexOperad[pow=\tau](2)$. However, we have $(\sigma,(12))\cdot \vg[graph]= \vg[graph]$ which implies that $C_2\to C_4\in \TransfersOf{\IncompleteIndexOperad[pow=\tau]}$.
\end{warning}

%% file: content/realizingtransfers.tex
\section{Realizing Transfer Systems}\label{section: realizing transfer systems}

In order to define the realization operad of a transfer system, we will extend the idea of support of a transfer to the complete realization operads $\Realize[pow=G]{M}$ by saying that $x\in \Realize[pow=G]{M}$ \emph{supports} a transfer $K\to H$ at position $i$ if there exists a twist map $\alpha\in\TwistMaps(G,\vn[finset])$ that structures a transfer $K\to H$ at position $i$ and for all $\vg[1],\vg[2]\in \Domain{\alpha}$, we have that \[x(\vg[1],\valpha{\vg[1]})\Intersect x(\vg[2],\valpha{\vg[2]}).\]

This leads to the following definition.
\begin{definition}
    Let $\tau\in \TransferSystemsOf{G}$ and $M$ a non-trivial intersection monoid. The \emph{$\tau$-incomplete realization by $M$} is the following $G$-symmetric sequence 
    \begin{align*}
        \Realize[pow=\tau]{M}(n) := \set[\bigg]{x\in \Realize[pow=G]{M}{n}\given \text{If }x\text{ supports a transfer }K\to H\text{, then }K\Transfer{\tau}H}.
    \end{align*}
\end{definition}

In this section we will first prove that $\Realize[pow=\tau]{M}$ is a well-defined $G$-suboperad of $\Realize[pow=G]{M}$, and then show that $\TransfersOf{\Realize[pow=\tau]{M}}=\tau$.

Let us first recall some ideas about poset valued operads. Given an operad of $G$-posets $P$, a \emph{lower} $G$-suboperad $Q \subseteq P$ is one such that each level $Q(n)$ is equivariantly downward closed inside $P(n)$. For the suboperads $\IncompleteIndexOperad[pow=\tau]$ of $\IncompleteIndexOperad[pow=G]$, since removing edges can't introduce any new complete subgraphs, we automatically obtain:

\begin{lemma}\label{lemma: incomplete are lower suboperads}
    The $G$-suboperads $\IncompleteIndexOperad[pow=\tau]$ are lower $G$-suboperads of $\IncompleteIndexOperad[pow=G]$.
\end{lemma}

We will follow the terminology of \cite{beuckelmannSmallCatalogueEnoperads2023}. Suppose we have a morphism of $G$-symmetric sequences $\phi:P\to Q$ where $P$ is a $G$-operad in sets, and $Q$ is a $G$-operad in posets. We call $\phi$ a \emph{lax operad morphism} if for any $x\in P(n)$, $y\in P(m)$, and $i\in \vn[finset]$, we have that \[\phi(gx\circ_i gy) \leq \phi(gx)\circ_i\phi(gy)\] for all $g\in G$. Given a lax morphism of $G$-operads $\phi:P\to Q$, and a lower $G$-suboperad $\vQ[~] \subseteq \vQ$, consider the following pullback of $G$-symmetric sequences:
\begin{center}
\begin{tikzpicture}[commutative diagrams/every diagram]
    \matrix[matrix of math nodes, name=m, commutative diagrams/every cell] {
    P & Q \\
    \widetilde{P} & \widetilde{Q} \\};

    \path[commutative diagrams/.cd, every arrow, every label]
    (m-2-1) edge node {$\widetilde{\phi}$} (m-2-2)
    edge[commutative diagrams/hook] (m-1-1)
    (m-1-1) edge node {$\phi$} (m-1-2)
    (m-2-2) edge[commutative diagrams/hook] (m-1-2);

    \node[draw=none, font=\Large] at ($(m-2-1)!0.2!(m-1-2)$) {$\urcorner$};
\end{tikzpicture}
\end{center}
Suppose we have $x\in \vP[~](n)$ and $y\in \vP[~](m)$. For any $i\in$, and $g\in G$, we have that \[ \phi(g\cdot x\circ_i g\cdot y) \leq \phi(g\cdot x)\circ_i\phi(g\cdot y) = \vphi[~](g\cdot x)\circ_i \vphi[~](g\cdot y)\in\vQ[~](m+n-1). \] Since $\vQ[~]$ is a lower suboperad, we see that $x\circ_i y\in \vP[~](m+n-1)$ and that the pull-back $\vP[~]$ is in fact a well-defined $G$-suboperad of $\vP$. We will record this observation as a lemma.

\begin{lemma}\label{lemma: pullback of lowersuboperads via lax are operads}
    Given a lax morphism of $G$-operads $\phi:P \to Q$ where $P\in\Operads(\Set^G)$ and $Q\in \Operads(\Poset^G)$. If $\vQ[~]$ is a lower $G$-suboperad of $\vQ$, then the pullback of the inclusion $\vQ[~]\hookrightarrow \vQ$ along $\phi$ in $G$-symmetric sequence $G$-symmetric is such that $\vP[~]:=\phi^\ast\vQ[~]$ is a $G$-suboperad of $\vP$. 
\end{lemma}

Consider the map (in $\SymSeq(\Set^G)$) $$\vp[graph]:\Realize[pow=G]{M}\to \IncompleteIndexOperad[pow=G]$$ where given $x\in \Realize[pow=G]{M}(n)$, $\vp[graph](x)$ is the graph which has an edge $\edge{(h_1,i_1)}{(h_2,i_2)}$ if the following two conditions hold:
\begin{enumerate}
    \item if $\vh[1]= \vh[2]$, then $\vi[1]\neq\vi[2]$, and 
    \item $x(h_1,i_2)\Intersect x(h_2,i_2)$.
\end{enumerate}

\begin{lemma}\label{lemma: incomplete realization is a pullback}
    The $G$-symmetric sequence $\Realize[pow=\tau]{M}$ is precisely the pull-back of the map $\vp[graph]$ along the inclusion $\iota:\IncompleteIndexOperad[pow=\tau]\hookrightarrow \IncompleteIndexOperad[pow=G]$.
\end{lemma}
\begin{proof}
    This is simply a matter of unpacking definitions.
    \begin{align*}
        \vp[pow=\ast,graph]\IncompleteIndexOperad[pow=\tau]{n} &= \set[\bigg]{x\in \Realize[pow=G]{M}{n}\given \vp[graph]{x}\in \IncompleteIndexOperad[pow=\tau]{n}} \\ 
        &= \set[\bigg]{x\in \Realize[pow=G]{M}{n}\given \text{if }\vp[graph]{x}\text{ supports a transfer }K\to H\text{, then }K\Transfer{\tau}H} \\
        &= \set[\bigg]{x\in \Realize[pow=G]{M}{n}\given \text{if }x\text{ supports a transfer }K\to H\text{, then }K\Transfer{\tau}H} \\
        &= \Realize[pow=\tau]{M}{n}
    \end{align*}
    Here we use that $\vp[graph]{x}$ supports a transfer $K\to H$, if and only if $x$ also supports a transfer $K\to H$ since edges in $\vp[graph]{x}$ correspond precisely to intersections between components of $x$.
\end{proof}

In view of \cref{lemma: incomplete are lower suboperads,lemma: pullback of lowersuboperads via lax are operads}, the following will then justify our claim that $\Realize[pow=\tau]{M}$ is a well-defined $G$-suboperad.

\begin{lemma}\label{lemma: projection onto graphs is lax}
    The map $\vp[graph]:\Realize[pow=G]{M}\to \IncompleteIndexOperad[pow=G]$ is a lax morphism of $G$-operads.
\end{lemma}

\begin{proof}
    Suppose we have $x\in \Realize[pow=G]{M}(n)$, $y\in \Realize[pow=G]{M}(m)$, and $i\in\vn[finset]$. Suppose $\vp[graph](x)\circ_i \vp[graph](y)$ has an edge $\edge{(h_1,j_1)}{(h_2,j_2)}$. Unpacking definitions, this is equivalent to the following three conditions holding: 
    \begin{enumerate}
        \item if $\vh[1]= \vh[2]$, then $\vi[1]\neq\vi[2]$,
        \item $x(h_1,\Combine[i,pow={n,m}]{j_1})\Intersect x(h_1,\Combine[i,pow={n,m}]{j_2})$, and
        \item if $\Combine[i,pow={n,m}]{j_1}=\Combine[i,pow={n,m}]{j_2}=i$, then $y(h_1,\Shift[i]{j_1})\Intersect y(h_2,\Shift[i]{j_2})$.
    \end{enumerate}
    In comparison, the graph $\vp[graph](x\circ_i y)$ has an edge $\edge{(h_1,j_1)}{(h_2,j_2)}$ if if $\vh[1]= \vh[2]$, then $\vi[1]\neq\vi[2]$, and (recalling \labelcref{equation: compact partial composition definition}) $$x(h_1,\Combine[i,pow={n,m}]{j_1})y(h_1,\Shift[oline,i]{j_1})\Intersect x(h_2,\Combine[i,pow={n,m}]{j_2})y(h_2,\Shift[oline,i]{j_2}).$$ 
    If this holds, by the axioms of an intersection monoid, we must have \[x(h_1,\Combine[i](j_1))\Intersect x(h_1,\Combine[i](j_1))\] for all $j_1,j_2$. 
    If $\Combine[i,pow={n,m}]{j_1}=\Combine[i,pow={n,m}]{j_2}=i$ then, since $x$ has strict columns, we must have that \[x(h_1,\Combine[i,pow={n,m}]{j_1})= x(h_1,\Combine[i,pow={n,m}]{j_2})\] and conclude \[y(h_1,\Shift[i]{j_1})\Intersect y(h_2,\Shift[i]{j_2}).\] 
    Hence, each edge of $\vp[graph](x\circ_i y)$ must also be contained in the graph $\vp[graph](x)\circ_i \vp[graph](y)$. i.e., $\vp[graph](x\circ_i y)\leq \vp[graph](x)\circ_i \vp[graph](y)$.
\end{proof}

\begin{theorem}
    Let $M$ be an intersection monoid, and $\tau\in\TransferSystemsOf{G}$. The $G$-symmetric sequence $\Realize[pow=\tau]{M}$ is a well-defined $G$-suboperad of $\Realize[pow=G]{M}$.
\end{theorem}
\begin{proof}
    This follows from \cref{lemma: incomplete are lower suboperads,lemma: pullback of lowersuboperads via lax are operads,lemma: incomplete realization is a pullback,lemma: projection onto graphs is lax}.
\end{proof}

\begin{remark}\label{remark:disconnection because of strict}
    \Cref{lemma: projection onto graphs is lax} is the reason we need to assume the columns of $\Realize[pow=G]{M}$ are strict. Enforcing strict columns in our construction has the unfortunate side effect of causing $\Realize[pow=G]{P(1)}{n}$ for many common operads $P$ such as the little $k$-cubes $\LittleCubeOperad[k]$ to be disconnected. 
    In particular, this condition stops $\Realize[pow=G]{\LittleCubeOperad[\infty](1)}{n}$ from being a model for $\mathbb{N}_\infty$-operads. 
    One can try to remove the requirement of strict columns in $\Realize[pow=\tau]{M}$; however, this will no longer be an operad in general. 
    We have included an explicit counterexample in \cref{appendix:strict columns} for the interested reader.
\end{remark}

Before moving onto showing that $\Realize[pow=\tau]{M}$ realizes the transfer system $\tau$, let us construct a map that will produce fixed points in the spaces $\Realize[pow=\tau]{M}{n}$. For a finite set $A$, we will write $$\Map^{\Disjoint}(A,M):= \set[\bigg]{f\in \Map(A,M)\given f(a)\Disjoint f(\va['])\text{ for all }\va\neq \va[']}.$$ 
This is always non-empty when $M$ is non-trivial by \cref{lemma: nontrivial have infinite disjoint elements}. Let $\Gamma \subseteq G\times \vn[finset]$ be a graph subgroup and consider the following composite: 
\begin{align*}
    \Psi_\Gamma:\Map^{\Disjoint}(\RightCosets{\Gamma}{G\times \vn[finset]},M) \hookrightarrow \Map(\RightCosets{\Gamma}{G\times \vn[finset]},M) \xrightarrow{q^\ast} \Map(G\times \vn[finset],M)= \OConstr[pow=G]{M}{n}
\end{align*} where $q$ is the quotient map. Given $x\in \Map^{\Disjoint}(\RightCosets{\Gamma}{G\times \vn[finset]},M)$, we have that $\Psi_{\Gamma}(x)(g,i)= x(\Gamma\cdot(g,i))$, and it follows that $\Psi_{\Gamma}(x)$ is fixed by $\Gamma$, is $\Sigma$-free and has strict columns. Hence, we have 
\begin{align*}
    \emptyset\neq\Image{\Psi_{\Gamma}} \subseteq \Realize[pow=G]{M}{n}^{\Gamma}.
\end{align*}

In fact, we have something stronger.

\begin{lemma}\label{lemma: constructed elements support only expected transfers}
    Let $M$ be a non-trivial intersection monoid and $\tau$ a transfer system of $G$. If $\Gamma \subseteq G\times \vn[finset]$ is a graph subgroup such that $\SetFromGraph[\Gamma]\in \GeneratedCoefficientSystemOf{\tau}{G/\vH[\Gamma]}$, then \[\emptyset\neq\Image{\Psi_{\Gamma}} \subseteq \Realize[pow=\tau]{M}{n}^{\Gamma}.\]
\end{lemma}

\begin{proof}
    We write $\Psi$ for $\Psi_\Gamma$, and suppose we have a decomposition \[\SetFromGraph[\Gamma]\cong \coprod_{k=1}^r \vH[\Gamma]/\vK_k.\] 
    We will show that if $\vp[graph]\Psi(x)$ supports a transfer $\vK[']\to \vH[']$, we have $\vH[']\leq \vH[\Gamma]$ and $\vK[']=\vH[']\cap \vK[pow=h,k]$ for some $h\in \vH[\Gamma]$ and $k\in\underline{r}$. 
    This is sufficient to prove the lemma as $\SetFromGraph[\Gamma]\in \GeneratedCoefficientSystemOf{\tau}{G/\vH[\Gamma]}$ and the decomposition of $\SetFromGraph[\Gamma]$ implies that $\vK[k]\Transfer{\tau} \vH[\Gamma]$ for each $k$. Since transfer systems are closed under restrictions and conjugations, we would then get $\vK[']\Transfer{\tau}\vH[']$.

    By construction, we have that $\Psi(x)(g_1,j_1)\Intersect \Psi(x)(g_2,j_2)$ if and only if $(g_2,j_2)\in \Gamma\cdot(g_1,j_1)$. It follows that $\vp[graph]\Psi(x)$ decomposes as \[\vp[graph]\Psi(x) = \coprod_{\Gamma\cdot (g,i)\in \RightCosets{\Gamma}{G\times \vn[finset]}} \CompleteGraph{\Gamma\cdot (g,i)}.\]
    Now, suppose that $\alpha:S\to \vn[finset]$ is a function that exhibits a transfer $\vK[']\to \vH[']$ in $\vp[graph]\vPsi{x}$. In particular, suppose $\alpha$ makes a $S/T$-complete subgraph in position $i_0$, and $g_0\in G$ is the element such that \[\vK[']\subseteq T\vg[0,inv] \subseteq S\vg[0,inv] =\vH[']\] where the first inclusion is maximal. 
    From the above graph decomposition, we must have \[\CompleteGraph{\Gamma(\alpha)} \subseteq \CompleteGraph{\Gamma\cdot (\vg[0],\vi[0])}.\]
    Which on vertices implies that \begin{align}\label{equation: twists of a fixed point inclusion}
        \set[\bigg]{(h,\valpha{h})\given h\in S} \subseteq \Gamma\cdot(\vg[0],\vi[0]).
    \end{align}
    The projection of $\Gamma\cdot(\vg[0],\vi[0])$ onto $G$ is $\vH[\Gamma]\vg[0]$, and so \labelcref{equation: twists of a fixed point inclusion} implies we have $S \subseteq \vH[\Gamma] \vg[0]$ which gives $S\vg[0,inv]=\vH['] \leq \vH[\Gamma]$. From the decomposition of $\SetFromGraph[\Gamma]$, we have for some $h\in \vH[\Gamma]$ and $k\in\underline{r}$, that $K_k^h=\Isotropy[\vH[\Gamma]]{\vi[0]}$ (using the action of $\vH[\Gamma]$ on $\vn[finset]$ via the action map $\phi_\Gamma$).  
    
    Consider the function $\beta:\vH[\Gamma]\vg[0]\to \vn[finset]$ given by $\beta(h)=\phi_{\Gamma}(h\vg[0,inv])(\vi[0])$. This function is exactly such that $\vbeta[inv]{\vi[0]}=K_k^{h}\vg[0]$ and we have \begin{align*}
        \Gamma\cdot(\vg[0],\vi[0])=\set[\bigg]{(h,\vbeta{h})\given h\in \vH[\Gamma]\vg[0]}.
    \end{align*}
    Comparing this with \cref{equation: twists of a fixed point inclusion}, we deduce that $\valpha = \restr{\vbeta}{S}$ and conclude that \begin{align*}
        \valpha[inv]{i}&= \vbeta[inv]{i}\cap S \\ 
        T&= \vK[pow=h,k]\vg[0]\cap S \\
        T\vg[0,inv]&= \vK[pow=h,k]\cap S\vg[0,inv] \\
        T\vg[0,inv]&= \vK[pow=h,k]\cap \vH['].
    \end{align*}
    Since $\alpha$ exhibits the transfer $\vK[']\to\vH[']$ via the $S/T$-complete subgraph, we have $\vK[']\subseteq T\vg[0,inv]$ is a maximal subgroup by definition and so conclude that $\vK[']=\vK[pow=h,k]\cap \vH[']$.
    Hence, we are done.
\end{proof}

We now have everything we need to prove the first main theorem of this paper is the following. Note that this implies the result stated in \cref{theorem A}.

\begin{theorem}\label{theorem: main realization}
    Let $\tau\in\TransferSystemsOf{G}$ and $M$ a non-trivial intersection monoid. The $G$-operad $\Realize[pow=\tau]{M}$ realizes the transfer system $\tau$.
\end{theorem}

\begin{proof}
    Firstly, it is straight forward to see that $\Realize[pow=\tau]{M}$ is $\Sigma$-free, and \cref{lemma: constructed elements support only expected transfers} shows that $\FixedSubgroupsOf[2]{\Realize[pow=\tau]{M}}$ contains the graph subgroups $H\times \set{\id}$, and $\tau \subseteq \TransfersOf{\Realize[pow=\tau]{M}}$. 

    To show that $\TransfersOf{\Realize[pow=\tau]{M}} \subseteq \tau$, it is enough to show that if $K\to H\notin \tau$, then $$\Realize[pow=\tau]{M}{n}^{\Gamma(H/K)}=\emptyset$$ where $n= \abs{H/K}$.
    Suppose $x\in \Realize[pow=\tau]{M}(n)^{\vGamma{H/K}}$.
    We must have $x$ is constant on any arbitrary orbit $\vGamma{H/K}\cdot (g,i)$. 
    This would mean that $\vp[graph](x)$ contains subgraphs $\CompleteGraph{\Gamma_{H/K}\cdot (g,i)}$ which would imply $\vp[graph](x)$ supports a $K\to H$ transfer. However, if $K\to H\notin \tau$, we can't have this by definition of $\Realize[pow=\tau]{M}$. Hence, by contradiction, we conclude $\TransfersOf{\Realize[pow=\tau]{M}} \subseteq \tau$ and are done.
\end{proof}

%% file: content/applications.tex
\section{A model for \texorpdfstring{$\mathbb{N}_\infty$-operads}{â„•âˆž-operads}.}\label{section: realizing ninf}
\subsection{Dyadic Intervals and Permutations}
In \cref{example: basic examples of intersection monoids} we defined the dyadic interval monoid $\DyadicIntervalMonoid$. As a reminder, this is the free monoid generated by two letters $\DyadicIntervalMonoid=F(\set{a,b})$, where for a general word \[w=\ell_1\ell_2\dots\ell_n\] with $\ell_k\in \set{a,b}$, we write the length by $\Length{w}=n$ and the $k$-th letter by $w^k:=\ell_k$. We put an intersection relation on $\DyadicIntervalMonoid$ where given two words $w_1,w_2\in \DyadicIntervalMonoid$, we set $w_1\Disjoint w_2$ if and only if there exists a $k\leq \min(\Length{\vw[1]},\Length{\vw[2]})$ such that $w_1^k\neq w_2^k$. 

The reason for the name of $\DyadicIntervalMonoid$ is that we can construct an intersection monoid map \[\alpha:\DyadicIntervalMonoid\to \LittleCubeOperad[1]{1}\]
where $\alpha(a)(z)=\frac{1}{2}z$ and $\alpha(b)(z)=\frac{1}{2}z+\frac{1}{2}$. This map is injective, and we can then identify $\DyadicIntervalMonoid$ with the intersection submonoid of $\LittleCubeOperad[1]{1}$ corresponding to embeddings with images the dyadic intervals $[k/2^n,k+1/2^n]$ for some $k,n$.

The corresponding intersection operad $\Realize[pow=\set{1}]{\DyadicIntervalMonoid}$ is related to the associative operad $\Assoc$. Write $\LinearOrderOperad{n}$ for the set of linear orders on $\vn[finset]$. This forms an operad isomorphic to the associative operad $\Assoc$ via the usual equivalence between linear orders on $\vn[finset]$ and permutations on $\vn[finset]$. For details on this equivalence, and verification that $\LinearOrderOperad$ forms an operad, see \cite{beuckelmannSmallCatalogueEnoperads2023}.

We can define an operad map \[\pi:\Realize[pow=\set{1}]{\DyadicIntervalMonoid}\to \LinearOrderOperad\] as follows. Given an element $w\in \Realize[pow=\set{1}]{\DyadicIntervalMonoid}{n}$, for each $i\neq j\in \vn[finset]$, since $w(i)\Disjoint w(j)$, there exists a minimal $k$ such that $w(i)_k\neq w(j)_k$ and so we either have $w(i)_k=a$ and $w(j)_k=b$, or $w(i)_k=b$ and $w(j)_k=a$. We can then define a linear order $<_w$ on $\vn[finset]$ where $i <_w j$ if $w(i)_k=a$ and $w(j)_k=b$. The map $\pi$ is then given by $\pi(w):=<_w$.

Because of this map $\pi$, and that the underlying monoid $\DyadicIntervalMonoid$ is free, we can view $\Realize[pow=\set{1}]{\DyadicIntervalMonoid}$ as a more ``free version'' of the associative operad. We find this operad interesting enough to give it a name, and will call it the \emph{Dyadic Associative Operad} $\DyadicAssoc$. 
\subsection{The Dyadic Barratt-Eccles Operad}

The non-equivariant Barrat-Eccles operad $\BarrattEccles$ is constructed from $\Assoc$ by first applying ``chaotic surgery'' $\widetilde{(-)}$ to each of its level spaces to get $\Assoc[~]$. Chaotic surgery is the process of constructing a simplicial set $\widetilde{X}$ from a set $X$ where the elements of the set form the vertices, and all possible edges and higher faces have also been included. Geometric realization is then applied to obtain $\BarrattEccles=\abs{\Assoc[~]}$. 

In comparison, instead of joining elements of the operad $\Assoc$ to enforce contractability, we will join elements of the monoid $\DyadicIntervalMonoid$ to enforce contractability. We will call the contracted version of $\DyadicIntervalMonoid$ the \emph{fat dyadic intervals monoid} and denote it by $\FatDyadicIntervalMonoid$. As we will show, the realization $\Realize[pow=\set{1}]{\FatDyadicIntervalMonoid}$ is an $\mathbb{E}_\infty$-operad which we will call the Dyadic Barratt-Eccles operad $\DyadicBarrattEccles$. 

Let us now construct $\FatDyadicIntervalMonoid$. Write $\CategoryOfFiniteSetInjective$ for the category with objects $\vn[finset]$ (including the empty set $\underline{0}=\emptyset$), and morphisms linear injective maps. Given $J \subseteq \vn[finset]$, we will write \[\delta_J:\underline{n-\abs{J}}\to \vn[finset]\] for the linear map that skips the values of $J$. There is a covariant functor \begin{align*}
    \vI[uline]: \CategoryOfFiniteSetInjective \to \Spaces 
\end{align*}
which on objects is $\vI[uline]{\vn[finset]}=I^n$ (and $\vI[uline]{\underline{0}}=\ast$), and given a morphism $\delta_J:\vm[finset]\to \vn[finset]$ where $m=n-\abs{J}$, the morphism $\vI[uline]{\delta_J}$ is the one determined by \[\left(\vI[uline]{\delta_J}(t_1,t_2,\dots, t_m)\right)_k= \begin{cases*}
    t_i& if $\delta_J(i)=k$ \\ 0 & otherwise.
\end{cases*}\]
In the case of $\delta_n: \underline{0}\to \vn[finset]$, the morphism $\vI[uline]{\delta_n}$ is the map $\ast\to I^n$ that picks out $(0,0,\dots,0)$. 

Let $\DyadicIntervalMonoid_n \subseteq \DyadicIntervalMonoid$ be the subset of words of length $n$. We also have a contravariant functor $\DyadicIntervalMonoid[uline]: \CategoryOfFiniteSetInjective^{op}\to \Spaces$ given by $\DyadicIntervalMonoid[uline]{\vn[finset]}=\DyadicIntervalMonoid_n$, topologized with the discrete topology, and on morphisms $\DyadicIntervalMonoid[uline]{\delta_J}:\DyadicIntervalMonoid[n]\to \DyadicIntervalMonoid[m]$ given by deleting the letters in each index determined by $J$. 
We will identify $\DyadicIntervalMonoid[n]$ with $\set{a,b}^n=\Map(\vn[finset],\set{a,b})$ and write words of length $n$ as vectors \[\vell[vec] = (\vell[1],\vell[2],\dots,\vell[n])\] where $\vell\in \set{a,b}$. 
The empty word will still be written as $e$, and if needed, we will write $w(\vell[vec])=\vell[1]\vell[2]\dots\vell[n]$ for the corresponding word in $\DyadicIntervalMonoid$. 
The morphism $\DyadicIntervalMonoid[uline](\vdelta[J])$ is then given by \[\DyadicIntervalMonoid[uline](\vdelta[J])(\vell[1],\vell[2],\dots \vell[n]) = (\vell[\vdelta{1}],\vell[\vdelta{2}],\dots, \vell[\vdelta{m}]).\]

\begin{definition}
    The \emph{fat dyadic intervals monoid} $\FatDyadicIntervalMonoid$ is the coend \[\FatDyadicIntervalMonoid:= \int^{\vn[finset]\in\CategoryOfFiniteSetInjective}\DyadicIntervalMonoid[uline]{\vn[finset]}\times \vI[uline]{\vn[finset]}.\]
\end{definition}

There is a natural continuous monoid structure on $\FatDyadicIntervalMonoid$ given by vector concatenation. In terms of representatives, this monoid structure is given by 
\[[(\vell[vec];\vt[vec])][(\vell[vec,'];\vt[vec,'])]=[(\vell[vec],\vell[vec,'];\vt[vec],\vt[vec,'])]\] where $\vell[vec],\vell[vec,']$ and $\vt[vec],\vt[vec,']$ is vector concatenation. The unit is the class $[(e,\ast)]$.

\begin{lemma}\label{lemma:fat dyadic is contractible}
    The monoid $\FatDyadicIntervalMonoid$ is contractible.
\end{lemma}

\begin{proof}
    For $\vt[vec]\in \vI[uline](\vn[finset])$, and $\lambda\in I$, we write \[\lambda \vt[vec] :=(\lambda \vt[1],\lambda \vt[2],\dots, \lambda \vt[n]).\] The obvious homotopies $\vH[n]:\vI[uline]{\vn[finset]}\times I\to \vI[uline]{\vn[finset]}$ given by $\vH[n](\vt[vec],\lambda)=(1-\lambda)\vt[vec]$ commutes with $\vI[uline]{\delta_J}$. 
    \begin{center}
    \begin{tikzpicture}[commutative diagrams/every diagram]
            \matrix[matrix of math nodes, name=m, commutative diagrams/every cell] {
            \vI[uline]{\vn[finset]}\times I & \vI[uline]{\vn[finset]}\\
            \vI[uline]{\vm[finset]}\times I & \vI[uline]{\vm[finset]}\\};

            \path[commutative diagrams/.cd, every arrow, every label]
            (m-2-1) edge node {$\vI[uline]{\delta_J}\times \id$} (m-1-1)
            (m-1-1) edge node {$\vH[n]$} (m-1-2)
            (m-2-2) edge node[swap] {$\vI[uline]{\delta_J}$} (m-1-2)
            (m-2-1) edge node {$\vH[m]$} (m-2-2);
        \end{tikzpicture}
    \end{center}
    As a consequence, this induces a homotopy \[H: \FatDyadicIntervalMonoid\times I\to \FatDyadicIntervalMonoid\] which on classes is given by \[H([(\vell[vec],\vt[vec])],\lambda)= [(\vell[vec],(1-\lambda)\vt[vec])].\] In particular, we find $H(-,0)=\id_{\FatDyadicIntervalMonoid}$ and $H([(\vell[vec],\vt[vec])],1)= [(\vell[vec],0\vt[vec])]= [(e,\ast)].$ Thus, we see $\FatDyadicIntervalMonoid$ is contractible.
\end{proof}

We will say that a representative $(\vell[vec];\vt[vec])$ is \emph{reduced} if $\vt[vec]=(\vt[1],\vt[2],\dots,\vt[n])$ is such that $\vt[k]\neq 0$ for all $k$. Every class $[(\vell[vec],\vt[vec])]$ has a unique reduced representative. We can use this to get a (discontinuous) monoid map \[\omega:\FatDyadicIntervalMonoid\to \DyadicIntervalMonoid\] where given $[(\vell[vec];\vt[vec])]\in \FatDyadicIntervalMonoid$, if $(\vell[vec];\vt[vec])=((\vell[1],\vell[2],\dots,\vell[n]);\vt[vec])$ is a reduced representative, then we set $\omega([(\vell[vec],\vt[vec])])=w(\vell[vec]) =\vell[1]\vell[2]\dots\vell[n]$. We can use the map $\omega$ to put an intersection relation on $\FatDyadicIntervalMonoid$. 

\begin{lemma}
    The monoid $\FatDyadicIntervalMonoid$ has an intersection monoid structure where given elements $x,y\in\FatDyadicIntervalMonoid$, we set $x\Intersect y$ if and only if $\omega(x)\Intersect \omega(y)$.
\end{lemma}

\begin{proof}
Suppose $\vx[1],\vx[2],\vy[1],\vy[2]\in \FatDyadicIntervalMonoid$ and $\vx[1]\vy[1]\Intersect\vx[2]\vy[2]$. Then we have by definition \[\omega(\vx[1]\vy[1])\Intersect\omega(\vx[2]\vy[2])=\omega(\vx[1])\omega(\vy[1])\Intersect\omega(\vx[2])\omega(\vy[2])\] in $\DyadicIntervalMonoid$. From the axioms, this implies $\omega(\vx[1])\Intersect\omega(\vx[2])$ and then $\vx[1]\Intersect\vx[2]$. 

Verification of the second axiom follows similarly.
\end{proof}

\subsection{\texorpdfstring{Realizing $\mathbb{N}_\infty$-operads}{Realizing â„•âˆž Operads}}

In \cref{remark:disconnection because of strict} we mentioned that the strict column condition causes the operads $\Realize[pow=\tau]{\LittleCubeOperad[k]{1}}$ to become level-wise disconnected. To illustrate why this occurs, suppose we have $x\in \LittleCubeOperad[k]{1}$ and a path $p\in \LittleCubeOperad[k]{1}^I$ where $p(0)=x$. We can't have a path where $x\Disjoint p(1)$ without first having $x \Intersect p(\lambda)$ and $x \neq p(\lambda)$ for some $\lambda\in I$. In other words, we can't continuously break up equal elements in the columns of $\Realize[pow=\tau]{\LittleCubeOperad[k]{1}}$ while keeping the columns strict. Hence, elements with different configuration of equal column elements must lie in different (path)-connected components. 

The monoid $\FatDyadicIntervalMonoid$ doesn't have this problem. For example, $[((a,b),(1,0))]=[((a),(1))]$, but $[((a,b),(1,t))]\Disjoint[][((a),(1))]$ for any $t>0$. This is the key idea of the following theorem.

\begin{theorem}\label{theorem:dyadic barrat eccles are ninf}
    Let $\tau\in\TransferSystemsOf{G}$. The $G$-operad $\Realize[pow=\tau]{\FatDyadicIntervalMonoid}$ is an $\mathbb{N}_\infty$-operad that realizes the transfer system $\tau$.
\end{theorem}
\begin{proof}
Set a graph subgroup $\Gamma \subseteq G\times \Sigma_n$ such that $\SetFromGraph[\Gamma]\in\GeneratedCoefficientSystemOf{\tau}$. We need to show that $\Realize[pow=\tau]{\FatDyadicIntervalMonoid}{n}^{\Gamma}$ is contractible. Recall the map $\Psi:=\Psi_{\Gamma}$ from \cref{lemma: constructed elements support only expected transfers} and fix an arbitrary element $[(\vp[vec],\vec{1})]\in \Image{\Psi} \subseteq \Realize[pow=\tau]{\FatDyadicIntervalMonoid}{n}^{\Gamma}$. Here $\vp[vec]\in \Map(G\times\vn[finset],\DyadicIntervalMonoid)$, and $\vec{1}$ is a function on $G\times\vn[finset]$ such that $\vec{1}(g,i)$ is a vector of $1$'s of the same length as $\vec{p}(g,i)$. We then interpret $[(\vp[vec],\vec{1})]$ to be the function on $G\times \vn[finset]$ given by \[[(\vp[vec],\vec{1})](g,i)= [(\vp[vec]{g,i},\vec{1}(g,i))].\]
Consider the map \begin{gather*}
    \Phi: \Map{G\times \vn[finset],\FatDyadicIntervalMonoid}\times I\to \Map{G\times \vn[finset],\FatDyadicIntervalMonoid} \\ \Phi([(\vell[vec];\vt[vec])],\lambda)(g,i) = [(\vp[vec]{g,i},\vell[vec]{g,i};\lambda\vec{1}(g,i),(1-\vlambda)\vt[vec]{g,i})].
\end{gather*}

To see that this map is well-defined, first observe that $\Map{G\times \vn[finset],\FatDyadicIntervalMonoid}$ is a topological monoid via component-wise multiplication. The map $\Phi$ is then the composition 
\begin{center}
    \begin{tikzpicture}[commutative diagrams/every diagram]
            \matrix[matrix of math nodes, name=m, commutative diagrams/every cell] {
            \Map(G\times\vn[finset],\FatDyadicIntervalMonoid)\times I \\
            I\times \Map(G\times\vn[finset],\FatDyadicIntervalMonoid) \\ I\times I\times \Map(G\times\vn[finset],\FatDyadicIntervalMonoid)\\ \Map(G\times\vn[finset],\FatDyadicIntervalMonoid)\times \Map(G\times\vn[finset],\FatDyadicIntervalMonoid)\\ \Map(G\times\vn[finset],\FatDyadicIntervalMonoid).\\};

            \path[commutative diagrams/.cd, every arrow, every label]
            (m-1-1) edge node {swap} (m-2-1)
            (m-2-1) edge node {$\Delta\times\id$} (m-3-1)
            (m-3-1) edge node {$[(\vp[vec],\lambda\vec{1})]\times H$}(m-4-1)
            (m-4-1) edge node {mult.}(m-5-1);
        \end{tikzpicture}
\end{center}
Here $H$ is the contraction from \cref{lemma:fat dyadic is contractible}.
This is such that $\Phi(-,0)$ is identity, and $\Phi(-,1)=[(\vp[vec],\vec{1})]$.

We claim the map $\Phi$ restricts to a map $\Realize[pow=\tau]{\FatDyadicIntervalMonoid}{n}^{\Gamma}\times I \to \Realize[pow=\tau]{\FatDyadicIntervalMonoid}{n}^{\Gamma}$. First, a quick verification shows that we get a restriction of the form $\Phi:\Realize[pow=G]{\FatDyadicIntervalMonoid}{n}^\Gamma\times I\to \Realize[pow=G]{\FatDyadicIntervalMonoid}{n}^\Gamma$. To show that we get a further restriction onto $\Realize[pow=\tau]{\FatDyadicIntervalMonoid}{n}^\Gamma$ consider the following. For any $(g,i)\in G\times \vn[finset]$, the reduced representative of \[\Phi(\lambda,[(\vell[vec];\vt[vec])])(g,i)=[(\vp[vec]{g,i},\vell[vec]{g,i};\lambda\vec{1},(1-\vlambda)\vt[vec]{g,i})]\] when $\lambda>0$ must contain $\vp[vec]{g,i}$ at the beginning of the letters component since $(\vec{p}(g,i),\vec{1}(g,i))$ is a reduced representation. So when we apply $\omega$, the corresponding word must contain $\omega(\vp[vec]{g,i})$ as a leading subword. 
As a consequence, if we have for some $(g_1,i_1),(g_2,i_2)\in G\times\vn[finset]$ that \[[(\vp[vec],\vec{1})](\vg[1],\vi[1])\Disjoint[]  [(\vp[vec],\vec{1})](\vg[2],\vi[2])\] then \[\Phi(\lambda,[(\vell[vec],\vt[vec])])(\vg[1],\vi[1])\Disjoint \Phi(\lambda,[(\vell[vec],\vt[vec])])(\vg[2],\vi[2])\] for any $\lambda>0$. This implies that any transfer $K\to H$ supported by $\Phi(\lambda,[(\vell[vec],\vt[vec])])$ for $\lambda>0$ must also be supported by $[(\vp[vec],\vec{1})]$. Hence, $\Phi(\lambda,[(\vell[vec],\vt[vec])])\in \Realize[pow=\tau]{\FatDyadicIntervalMonoid}(n)^{\Gamma}$ when $\lambda>0$, and the map $\Phi$ restricts as claimed.

We have thus constructed a contraction via d$\Phi$ and so $\Realize[pow=\tau]{WF}{n}^{\Gamma}$ is contractible. Hence, the theorem has been proven.
\end{proof}

Since $\Coinduced[pow=\tau,e,uline]{\DyadicBarrattEccles}= \Realize[pow=\tau]{\FatDyadicIntervalMonoid}$, \cref{theorem:dyadic barrat eccles are ninf} gives us the promised \cref{theorem B} as a consequence.

%% file: content/appendix.tex
\section{Proof that the Incomplete Indexing Sequence is an Operad}\label[appendix]{appendix:indexing_is_an_operad}
In this appendix we will go through the details of verifying that $\IncompleteIndexOperad[pow=G]$ is a well-defined operad in $\Poset^G$.

\begin{lemma}
    The $\circ_i$-composition morphisms as defined in \cref{definition:incomplete_indexing_operad} are well-defined morphisms in $\Poset^G$.
\end{lemma}
\begin{proof}
    Let $\vg[graph]\in\IncompleteIndexOperad[pow=G]{n}$, $\vk[graph]\in\IncompleteIndexOperad[pow=G]{m}$, and $i\in\vn[finset]$. For $g\in G$, the graph $g\cdot (\vg[graph]\circ_i \vk[graph])$ has an edge $$\edge{(h_1,j_1)}{(h_2,j_2)}$$ if $\vg[graph]\circ_i \vk[graph]$ has the edge $$\edge{(\vg[inv]h_1,j_1)}{(\vg[inv]h_2,j_2)}.$$ 
    By definition, this means $\vg[graph]$ has the edge $$\edge{(\vg[inv]h_1,\Combine[i]{j_1})}{(\vg[inv]h_2,\Combine[i]{j_2})}$$ and, if $j_1,j_2\in \Combine[i,inv](i)$, then $\vk[graph]$ has edge $$\edge{(\vg[inv]\vh[1],\Shift[i]{j_1})}{(\vg[inv]\vh[2],\Shift[i]{j_2})}.$$
    i.e., $g\cdot \vg[graph]$ has edge $\edge{(h_1,\Combine[i]{j_1})}{(h_2,\Combine[i]{j_2})}$ and $g\cdot \vk[graph]$ has edge $\edge{(\vh[1],\Shift[i]{j_1})}{(\vh[2],\Shift[i]{j_2})}$. 
    Hence, we conclude that \[(g\cdot \vg[graph])\circ_i (g\cdot \vk[graph])=g\cdot (\vg[graph]\circ_i\vk[graph]).\]Moreover, adding edges to either $\vg[graph]$ or $\vk[graph]$ will clearly cause $\vg[graph]\circ_i\vk[graph]$ to have more edges. We then deduce that $\circ_i$ are well-defined morphisms in $\Poset^G$. 
\end{proof}

\begin{lemma}
    The $\circ_i$-composition maps of $\IncompleteIndexOperad[pow=G]$ satisfies the associativity condition.
\end{lemma}
\begin{proof}
    To prove associativity, suppose we have graphs $\vg[graph]\in\IncompleteIndexOperad[pow=G]{n},\vh[graph]\in \IncompleteIndexOperad[pow=G]{m}$, and $\vk[graph]\in\IncompleteIndexOperad[pow=G]{\ell}$, and integers $i\in\vn[finset]$, and $j\in \underline{n+m-1}$. We must show the following holds:
    \begin{gather}\label{eq:graph_associativity}
    (\vg[graph]\circ_i \vh[graph])\circ_j \vk[graph] = \begin{cases*}
        (\vg[graph] \circ_j \vk[graph]) \circ_{i+\ell-1} \vh[graph] & if $1\leq j < i$, \\ 
        \vg[graph]\circ_i (\vh[graph] \circ_{j-i+1} \vk[graph]) & if $i\leq j \leq i+m-1$, \\
        (\vg[graph]\circ_{j-m+1} \vk[graph]) \circ_{i} \vh[graph] & if $i+m\leq j \leq n+m-1$.  
    \end{cases*}
\end{gather} 
    Unpacking definitions, the composite $(\vg[graph]\circ_i\vh[graph])\circ_j\vk[graph]$ has an edge $\edge{(h_1,j_1)}{(h_2,j_2)}$ if the following three conditions holds:
    \begin{gather}
        \text{the graph $\vg[graph]$ has an edge $\edge{(h_1,\Combine[i,pow={n,m}]{\Combine[j,pow={n+m-1,\ell}]{j_1}})}{(h_2,\Combine[i,pow={n,m}]{\Combine[j,pow={n+m-1,\ell}]{j_2}})}$} \label{eq:graph_assoc_condition_1}\\ 
        \begin{gathered}
        \text{if $\Combine[i,pow={n,m}]{\Combine[j,pow={n+m-1,\ell}]{\vj[1]}}=\Combine[i,pow={n,m}]{\Combine[j,pow={n+m-1,\ell}]{\vj[2]}}=i$, the graph $\vh[graph]$ has an edge} \\ \edge{(h_1,\Shift[i]{\Combine[j,pow={n+m-1,\ell}]{\vj[1]}})}{(h_2,\Shift[i]{\Combine[j,pow={n+m-1,\ell}]{\vj[2]}})}
        \end{gathered} \label{eq:graph_assoc_condition_2}\\ 
        \begin{gathered}
            \text{if $\Combine[j,pow={n+m-1,\ell}]{\vj[1]}=\Combine[j,pow={n+m-1,\ell}]{\vj[2]}=j$, the graph $\vk[graph]$ has an edge } \\ \edge{(h_1,\Shift[j]{j_1})}{(h_2,\Shift[j]{j_2})}. 
        \end{gathered}\label{eq:graph_assoc_condition_3}
    \end{gather}

    \underline{\emph{Case 1: $1\leq j<i$.}}
    In this case we have \[\Combine[i,pow={n,m}]\circ\Combine[j,pow={n+m-1,\ell}]=\Combine[j,pow={n,\ell}]\circ\Combine[i+\ell-1,pow={n+\ell-1,m}],\] and condition \labelcref{eq:graph_assoc_condition_1} above is equivalent to:
    \begin{equation}
            \text{the graph $\vg[graph]$ has an edge $\edge{(h_1,\Combine[j,pow={n,\ell}]{\Combine[i+\ell-1,pow={n+\ell-1,m}]{j_1}})}{(h_2,\Combine[i,pow={n,m}]{\Combine[j,pow={n+m-1,\ell}]{j_2}})}$.}
         \label{eq:graph_assoc_firstcase_condition_1}
    \end{equation}
    Let $k\in \underline{n+m+\ell-2}$. Observe that as $1\leq j < i$, we have the following sequence of equivalences
    \begin{gather*}
        \Combine[i,pow={n,m}]{\Combine[j,pow={n+m-1,\ell}]{k}}=i \\
        \iff \Combine[j,pow={n,\ell}]{\Combine[i+\ell-1,pow={n+\ell-1,m}](k)}=i \\ 
        \iff \Combine[i+\ell-1,pow={n+\ell-1,m}](k)=i+\ell-1.
    \end{gather*}
    Also, for $k\geq i+\ell-1$, we have that \[\Shift[i]\Combine[j,pow={n+m-1,\ell}]{k}=\Shift[i+\ell-1]{k}.\]
    Hence, condition \labelcref{eq:graph_assoc_condition_2} above is equivalent to:
    \begin{equation}
        \begin{gathered}
            \text{if $\Combine[i+l-1,pow={n+\ell-1,m}]{j_1}=\Combine[i+l-1,pow={n+\ell-1,m}]{j_2}=i+\ell-1$, the graph $\vh[graph]$ has an edge} \\ \edge{(h_1,\Shift[i+\ell-1]{j_1})}{(h_2,\Shift[i+\ell-1]{j_2})}.
        \end{gathered} \label{eq:graph_assoc_firstcase_condition_2}
    \end{equation}
    Similarly, for $1\leq j <i$ we have the following equivalences
    \begin{gather*}
        \Combine[j,pow={n+m-1,\ell}]{k}=j \\
        \iff \Combine[i,pow={n,m}]{\Combine[j,pow={n+m-1,\ell}]{k}}=j \\ 
        \iff \Combine[j,pow={n,\ell}]{\Combine[i+\ell-1,pow={n+\ell-1,m}](k)}=j.
    \end{gather*}
    Moreover, since $\Combine[i+\ell-1,pow={n+\ell-1,m}](k)=k$ for all $k<i+\ell-1$, if $\Combine[j,pow={n,\ell}]{\Combine[i+\ell-1,pow={n+\ell-1,m}](k)}=j$, then we must have $\Combine[i+\ell-1,pow={n+\ell-1,m}]{k}=k$ when $1\leq j < i$. Hence, \labelcref{eq:graph_assoc_condition_3} is equivalent to:
    \begin{equation}
        \begin{gathered}
            \text{if $\Combine[i,pow={n,m}]{\Combine[j,pow={n+m-1,\ell}]{\vj[1]}}=\Combine[i,pow={n,m}]{\Combine[j,pow={n+m-1,\ell}]{\vj[2]}}=j$, the graph $\vk[graph]$ has an edge} \\ \edge{(h_1,\Shift[j]{\Combine[i+\ell-1,pow={n+\ell-1,m}]{j_1}})}{(h_2,\Shift[j]{\Combine[i+\ell-1,pow={n+\ell-1,m}]{j_2}})}.
        \end{gathered} \label{eq:graph_assoc_firstcase_condition_3}
    \end{equation}
    The statements \labelcref{eq:graph_assoc_firstcase_condition_1,eq:graph_assoc_firstcase_condition_2,eq:graph_assoc_firstcase_condition_3} are exactly the conditions for $(\vg[graph]\circ_j\vk[graph])\circ_{i+\ell-1}\vh[graph]$ to have the edge $\edge{(h_1,j_1)}{(h_2,j_2)}$. Hence, we have shown that when $1\leq j < i$ that \[(\vg[graph]\circ_i \vh[graph])\circ_j \vk[graph] = (\vg[graph]\circ_j\vk[graph])\circ_{i+\ell-1}\vh[graph].\]

    \underline{\emph{Case 2: $i\leq j \leq i+m-1$.}}
    Unpacking definitions, the graph $\vg[graph]\circ_i(\vh[graph]\circ_{j-i+1}\vk[graph])$ has an edge $\edge{(\vh[1],\vj[1])}{(\vh[2],\vj[2])}$ if the following three conditions hold.
    \begin{equation}
        \text{The graph $\vg[graph]$ has the edge } \edge{(\vh[1],\Combine[i,pow={n,m+\ell-1}]{\vj[1]})}{(\vh[2],\Combine[i,pow={n,m+\ell-1}]{\vj[2]})} \label{eq:graph_assoc_secondcase_condition_1}
    \end{equation}
    \begin{equation}
        \begin{gathered}
            \text{If $\Combine[i,pow={n,m+\ell-1}]{j_1}=\Combine[i,pow={n,m+\ell-1}]{j_2}=i$, then the graph $\vh[graph]$ has edge} \\ \edge{(\vh[1],\Combine[j-i+1,pow={m,\ell}]{\Shift[i]{j_1}})}{(\vh[2],\Combine[j-i+1,pow={m,\ell}]{\Shift[i]{j_2}})}. \label{eq:graph_assoc_secondcase_condition_2}
        \end{gathered}
    \end{equation}
    \begin{equation}
        \begin{gathered}
            \text{If $\Combine[i,pow={n,m+\ell-1}]{j_1}=\Combine[i,pow={n,m+\ell-1}]{j_2}=i$, and $\Combine[j-i+1,pow={m,\ell}]{\Shift[i]{\vj[1]}}=\Combine[j-i+1,pow={m,\ell}]{\Shift[i]{\vj[2]}}=j-i+1$}\\ \text{then the graph $\vk[graph]$ has edge } \edge{(\vh[1],\Shift[j-i+1]{\Shift[i]{j_1}})}{(\vh[2],\Shift[j-i+1]{\Shift[i]{j_2}})}. \label{eq:graph_assoc_secondcase_condition_3}
        \end{gathered}
    \end{equation}

    In this case we have \[\Combine[i,pow={n,m}]\circ \Combine[j,pow={n+m-1,\ell}]= \Combine[i,pow={n,m+\ell-1}]\] and so \labelcref{eq:graph_assoc_condition_1} and $\labelcref{eq:graph_assoc_secondcase_condition_1}$ are equivalent. Also, note that for $k> i$ we have that \[\Shift[i]{\Combine[j,pow={n+m-1,\ell}]{k}}= \Combine[j-i+1,pow={m,\ell}]{\Shift[i]{k}}.\] Hence, we deduce that $\labelcref{eq:graph_assoc_condition_2}$ and $\labelcref{eq:graph_assoc_secondcase_condition_2}$ are equivalent. This equation also implies that
    \begin{gather*}
       \Combine[j-i+1,pow={m,\ell}]{\Shift[i]{k}}=j-i+1=\Shift[i]{j} \\ 
       \iff \Combine[j,pow={n+m-1,\ell}]{k}=j \\ 
       \implies \Combine[i,pow={n,m}]{\Combine[j,pow={n+m-1,\ell}]{k}}= \Combine[i,pow={n,m+\ell-1}]{k} = i 
    \end{gather*}
    Since $\Shift[j-i+1]\circ \Shift[i] = \Shift[j]$, we deduce the statements $\labelcref{eq:graph_assoc_condition_3}$ and $\labelcref{eq:graph_assoc_secondcase_condition_3}$ are equivalent when $i\leq j\leq i+m-1$. Thus, we have shown that \labelcref{eq:graph_assoc_condition_1,eq:graph_assoc_condition_2,eq:graph_assoc_condition_3} are equivalent to \labelcref{eq:graph_assoc_secondcase_condition_1,eq:graph_assoc_secondcase_condition_2,eq:graph_assoc_secondcase_condition_3}, and so we have \[(\vg[graph]\circ_i \vh[graph])\circ_j \vk[graph]=\vg[graph]\circ_i (\vh[graph] \circ_{j-i+1} \vk[graph])\] when $i\leq j\leq i+m-1$.

    \underline{\emph{Case 3: $i+m\leq j \leq n+m-1$.}}
    This case is similar to Case 1, except for some differences in indexing, and so we omit it.
\end{proof}

\begin{lemma}
    The $\circ_i$-composition maps of $\IncompleteIndexOperad[pow=G]$ satisfies the equivariance condition.
\end{lemma}
\begin{proof}
    For graphs $\vg[graph]\in \IncompleteIndexOperad[pow=G]{n}$, and $\vk[graph]\in \IncompleteIndexOperad[pow=G]{m}$, integer $i\in\vn[finset]$, and permutations $\sigma\in \vSigma[n]$, and $\tau\in \vSigma[m]$, we need to show that 
    \begin{equation}
        (\sigma\cdot \vg[graph])\circ_{\vsigma{i}}(\tau\cdot\vk[graph]) = (\vsigma\circ_i\vtau)\cdot (\vg[graph]\circ_i\vk[graph]) \label{eq:indexing_equivariance_goal}
    \end{equation} 
    holds. The graph $(\sigma\cdot \vg[graph])\circ_{\vsigma{i}}(\tau\cdot\vk[graph])$ has the edge $\edge{(\vh[1],\vj[1])}{(\vh[2],\vj[2])}$ if and only if 
    \begin{gather}
        \text{the graph $\vg[graph]$ has the edge $\edge{(\vh[1],\vsigma[inv]\Combine[\vsigma{i},pow={n,m}]{\vj[1]})}{(\vh[2],\vsigma[inv]\Combine[\vsigma{i},pow={n,m}]{\vj[2]})}$, and} \label{eq:indexing_equivariance_left_condition_1} \\ \begin{gathered}
            \text{if $\Combine[\vsigma{i},pow={n,m}]{\vj[1]}=\Combine[\vsigma{i},pow={n,m}]{\vj[2]}=\vsigma{i}$, the graph $\vk[graph]$ has the edge} \\ \edge{(\vh[1],\vtau[inv]\Shift[\vsigma{i}]{\vj[1]})}{(\vh[2],\vtau[inv]\Shift[\vsigma{i}]{\vj[2]})}. 
        \end{gathered}\label{eq:indexing_equivariance_left_condition_2}
    \end{gather}
    Whereas, the graph $(\vsigma\circ_i\vtau)\cdot (\vg[graph]\circ_i\vk[graph])$ has the edge $\edge{(\vh[1],\vj[1])}{(\vh[2],\vj[2])}$ if and only if
    \begin{gather}
        \begin{gathered}
        \text{the graph $\vg[graph]$ has the edge} \\  \edge{(\vh[1],\Combine[i,pow={n,m}]{(\vsigma[inv]\circ_{\sigma(i)}\vtau[inv])(\vj[1])})}{(\vh[2],\Combine[i,pow={n,m}]{(\vsigma[inv]\circ_{\sigma(i)}\vtau[inv])(\vj[2])})}
        \end{gathered} \label{eq:indexing_equivariance_right_condition_1}\shortintertext{and, }
         \begin{gathered}
            \text{if $\Combine[i,pow={n,m}]{(\vsigma[inv]\circ_{\sigma(i)}\vtau[inv])(\vj[1])}=\Combine[i,pow={n,m}]{(\vsigma[inv]\circ_{\sigma(i)}\vtau[inv])(\vj[2])}=i$, the graph $\vk[graph]$ has the edge} \\ \edge{(\vh[1],\Shift[i]{(\vsigma[inv]\circ_{\sigma(i)}\vtau[inv])(\vj[1])})}{(\vh[2],\Shift[i]{(\vsigma[inv]\circ_{\sigma(i)}\vtau[inv])(\vj[2])})}.
        \end{gathered} \label{eq:indexing_equivariance_right_condition_2}
    \end{gather}
    Recall from \cref{lemma:O_construction_is_an_operad} we have the equations
    \begin{gather*}
        \vsigma[inv]\Combine[\vsigma{i}]{k} = \Combine[i]{(\vsigma[inv]\circ_{\vsigma{i}}\id)(k)}  = \Combine[i]{(\vsigma[inv]\circ_{\vsigma{i}}\vtau[inv])(k)} \shortintertext{and,}
        \vtau[inv]\Shift[\vsigma{i}](k) = \Shift[\vsigma{i}]((\id\circ_{\sigma(i)}\vtau[inv])(k)) = \Shift[i]((\vsigma[inv]\circ_{\sigma(i)}\vtau[inv])(k)).
    \end{gather*}
    Hence, we have that $\labelcref{eq:indexing_equivariance_left_condition_1}$ is equivalent to $\labelcref{eq:indexing_equivariance_right_condition_1}$, and $\labelcref{eq:indexing_equivariance_left_condition_2}$ is equivalent to $\labelcref{eq:indexing_equivariance_right_condition_2}$. Hence, \labelcref{eq:indexing_equivariance_goal} holds and we are done.
\end{proof}

\section{Strict Columns are Necessary}\label[appendix]{appendix:strict columns}

In this small appendix, we will illustrate why we need to assume the strict column condition in the definition of $\Realize[pow=G]{M}$. 
Let us write $\OConstr[pow=G,\Sigma]{M}$ for the suboperad of $\OConstr[pow=G]{M}$ of all $\Sigma$-free elements, and define a $G$-symmetric sequence \[\OConstr[pow=\tau,\Sigma]{M}(n):=\set[\bigg]{x\in \OConstr[pow=G,\Sigma]{M}(n)\given \text{if } x \text{ supports a transfer }K\to H,\text{ then }K\Transfer{\tau}H}.\] 
That is, $\OConstr[pow=\tau,\Sigma]{M}$ is $\Realize[pow=\tau]{M}$ without the requirement the columns are strict. We will show that $\OConstr[pow=\tau,\Sigma]$ fails to be an operad generally.

Consider the case $G=C_4$, $\tau=\set[\big]{C_2\to C_4}$, and $M=\LittleCubeOperad[1](1)$. Consider the element $x\in \OConstr[pow=\tau,\Sigma]{\LittleCubeOperad[1]{1}}{2}$ where each $x(g,i):I\to I$ is given by the following: \begin{gather*}
    x(e,1)(z)=\frac{1}{4}z+0, \quad x(e,2)(z)=\frac{1}{4}z+\frac{3}{4}, \\  
    x(\tau^2,1)(z)=\frac{1}{4}z+\frac{1}{8}, \quad x(\tau^2,2)(z)=\frac{1}{4}z+\frac{5}{8}, \\
        x(\tau,1)(z)=\frac{1}{4}z+\frac{3}{4}, \quad x(\tau,2)(z)=\frac{1}{4}z+0, \\ 
        x(\tau^3,1)(z)=\frac{1}{4}z+\frac{5}{8}, \quad x(\tau^3,2)(z)=\frac{1}{4}z+\frac{1}{8}.
\end{gather*}
Similarly, let $y\in \OConstr[pow=\tau,\Sigma]{\LittleCubeOperad[1]{1}}{2}$ be given by the following maps:
\begin{gather*}
    y(e,1)(z)=\frac{1}{2}z+0, \quad y(e,2)(z)=\frac{1}{2}z+\frac{1}{2}, \\  
    y(\tau^2,1)(z)=\frac{1}{2}z+0, \quad y(\tau^2,2)(z)=\frac{1}{2}z+\frac{1}{2}, \\
        y(\tau,1)(z)=\frac{1}{2}z+\frac{1}{2}, \quad y(\tau,2)(z)=\frac{1}{2}z+0, \\ 
        y(\tau^3,1)(z)=\frac{1}{2}z+\frac{1}{2}, \quad y(\tau^3,2)(z)=\frac{1}{2}z+0.
\end{gather*}
The element $z=x\circ_1 y$ is then 
\begin{gather*}
    z(e,1)(z)=\frac{1}{8}z+0, \quad z(e,2)(z)=\frac{1}{8}z+\frac{1}{8}, \quad z(e,3)(z)= \frac{1}{4}z+\frac{3}{4},\\  
    z(\tau^2,1)(z)=\frac{1}{8}z+\frac{1}{8}, \quad z(\tau^2,2)(z)=\frac{1}{8}z+\frac{2}{8}, \quad z(\tau^2,3)(z)= \frac{1}{4}z+\frac{5}{8},\\
    z(\tau,1)(z)=\frac{1}{8}z+\frac{7}{8}, \quad z(\tau,2)(z)=\frac{1}{8}z+\frac{6}{8}, \quad z(\tau,3)(z)= \frac{1}{4}z+0,\\
    z(\tau^3,1)(z)=\frac{1}{8}z+\frac{6}{8}, \quad z(\tau^3,2)(z)=\frac{1}{8}z+\frac{5}{8}, \quad z(\tau^3,3)(z)= \frac{1}{4}z+\frac{1}{8}.
\end{gather*}
The twist map $\alpha:C_2\to \underline{3}$ with $\alpha(e)=2$, $\alpha(\tau^2)=1$ then exhibits a transfer of the form $e\to C_2$ which is not in $\tau$. Hence, we see that $\OConstr[pow=\tau,\Sigma]{\LittleCubeOperad[1]{1}}$ is not closed under composition.